\documentclass[pdflatex,sn-mathphys-num,twocolumn]{sn-jnl}

\usepackage{graphicx}%
\usepackage{commands}
\usepackage{multirow}%
\usepackage{amsmath,amssymb,amsfonts}%
\usepackage{amsthm}%
\usepackage{mathrsfs}%
\usepackage[title]{appendix}%
\usepackage[export]{adjustbox}
\usepackage{textcomp}%
\usepackage{manyfoot}%
\usepackage{booktabs}%
\usepackage{algorithm}%
\usepackage{algorithmicx}%
\usepackage{algpseudocode}%
\usepackage{listings}%
\usetikzlibrary{arrows.meta,positioning,fit}

\theoremstyle{thmstyleone}%
\newtheorem{theorem}{Theorem}%
\newtheorem{proposition}[theorem]{Proposition}%
\newtheorem{lemma}[theorem]{Lemma}%
\newtheorem{corollary}[theorem]{Corollary}

\theoremstyle{thmstyletwo}%
\newtheorem{remark}{Remark}%

\theoremstyle{thmstylethree}%
\newtheorem{definition}[theorem]{Definition}%

\begin{document}

\title[Diffusion-Shock PDEs for Deep Learning on Position-Orientation Space]{Diffusion-Shock PDEs for Deep Learning on Position-Orientation Space}

\author*[1]{\fnm{Finn M.} \sur{Sherry}}\email{f.m.sherry@tue.nl}

\author[2]{\fnm{Kristina} \sur{Schaefer}}\email{schaefer@mia.uni-saarland.de}

\author[1]{\fnm{Remco} \sur{Duits}}\email{r.duits@tue.nl}

\affil*[1]{\orgdiv{CASA \& EAISI, Dept. of Mathematics \& Computer Science}, \orgname{Eindhoven University of Technology}, \orgaddress{\city{Eindhoven}, \country{the Netherlands}}}

\affil[2]{\orgdiv{Mathematical Image Analysis Group, Department of Mathematics and Computer Science}, \orgname{Saarland University}, \orgaddress{\city{Saarbrücken}, \country{Germany}}}

\abstract{
We extend Regularised Diffusion-Shock (RDS) filtering from Euclidean space $\Rtwo$ \cite{schaefer2024regularisedds} to position-orientation space $\M \cong \Rtwo \times S^1$.
This has numerous advantages, e.g. making it possible to enhance and inpaint crossing structures, since they become disentangled when lifted to $\M$. 
We create a version of the algorithm using gauge frames to mitigate issues caused by lifting to a finite number of orientations. This leads us to study generalisations of diffusion, since the gauge frame diffusion is not generated by the Laplace-Beltrami operator. 
RDS filtering compares favourably to existing techniques such as Total Roto-Translational Variation (TR-TV) flow \cite{smets2021tvflow,chambolle2019tv}, NLM \cite{nlm}, and BM3D \cite{bm3dold} when denoising images with crossing structures, particularly if they are segmented. Furthermore, we see that $\M$ RDS inpainting is indeed able to restore crossing structures, unlike $\Rtwo$ RDS inpainting. \\[4pt]
In addition to the contributions of our SSVM submission \cite{Sherry2025DiffusionShock}, in this extended work we provide new theorical results and automate RDS filtering by integrating it into a geometric deep learning framework. Regarding our theoretical contributions, we prove that our generalised diffusions are still well-posed, smoothing, and analytic. We developed an RDS filtering PDE layer for the PDE-CNN and PDE-G-CNN deep learning frameworks, using a novel gating mechanism.
We show that these new RDS PDE layers can be beneficial in various impainting and denoising tasks.
}


\maketitle

\section{Introduction}\label{sec:introduction}
PDE-based image processing techniques have been studied  for decades and successfully employed in myriad applications including image analysis, denoising, and inpainting (e.g. \cite{Sapiro2006GeometricAnalysis,AGLM93,Lindeberg1994ScaleVision,Lindeberg1994ScaleScales,Schnorr1994UniqueFunctionals,Morel1995VariationalSegmentation,Weickert1998AnisotropicProcessing,Papafitsoros2014CombinedReconstruction,Chambolle2004AlgorithmApplications,Osher1988FrontsFormulations}), while providing clear geometric interpretations. For example, many of these PDE-based methods are inherently equivariant, preserving the symmetries (to e.g. roto-translations) of the plane.
Within this class, Regularised Diffusion-Shock (RDS) filtering, recently developed by Schaefer \& Weickert, is a provably stable and highly effective method \cite{SW23,schaefer2024regularisedds}.

\begin{figure*}
\centering
\begin{tikzpicture}
\node[anchor=south west, inner sep=0] (image) at (0, 0) {\includegraphics[width=\textwidth]{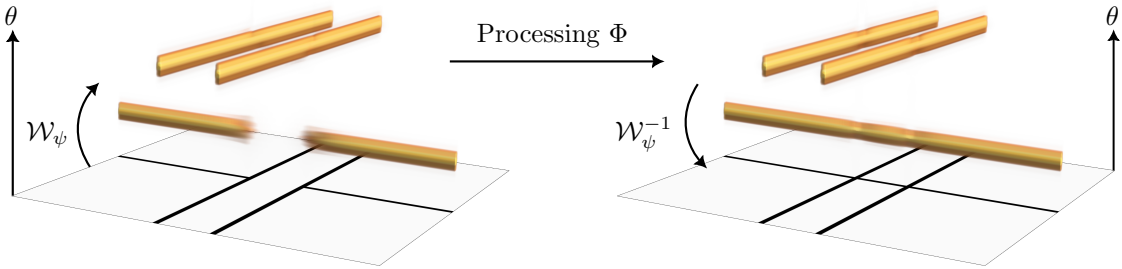}};
\begin{scope}[x={(image.south east)}, y={(image.north west)}, shift={(0.5, 0.5)}, scale=0.5]
\draw(-0.83, -0.1)node{\contour{white}{$\W_\psi$}};
\draw(-0.89, 0.6)node{$\theta$};
\draw(0.0, 0.48)node{Processing $\Phi$};
\draw(0.15, -0.1)node{\contour{white}{$\W_\psi^{-1}$}};
\draw(0.92, 0.6)node{$\theta$};
\end{scope}
\end{tikzpicture}
\caption{Performing multi-orientation processing. Lifting disentangles crossing and overlapping structures.}
\label{fig:multiorientation_processing}
\end{figure*}

Many of these PDE methods have benefited from being lifted to position-orientation space $\M \cong \Rtwo \times S^1$; here we mention just a few successful examples \cite{Citti2016mcf,cittisartiJMIV,Batard2019ConnectionPerception,Boscain2018CorticalDiffusion,DuitsJMIV2013,Bekkers2014MultiTracking,chambolle2019tv,smets2021tvflow}. Processing in position-orientation space has numerous advantages. Most significantly, crossing structures become disentangled by lifting (see Fig.~\ref{fig:multiorientation_processing}), which helps preserve crossings and corners during the PDE evolution. Orientation-aware processing, such as contour completion, also becomes more straightforward, since orientation is encoded in the domain. 
The authors extended RDS filtering to position-orientation space $\M$ in \cite{Sherry2025DiffusionShock}, combining the edge-preserving denoising capabilities of RDS filtering with the benefits of processing with PDEs on $\M$, such as preservation of crossings. This paper is a substantially extended journal version of the aforementioned article \cite{Sherry2025DiffusionShock}, which was presented at the 10th International Conference on Scale Space and Variational Methods in Computer Vision (SSVM 2025); we here i.a. present new theoretical results (e.g. Thm.~\ref{thm:data_driven_lie_cartan_laplacian_sectorial}, Thm.~\ref{thm:stability}) regarding RDS filtering on $\M$, and moreover automate RDS filtering via PDE-based deep learning as we will explain next.

In the past decade, deep learning methods have had a large impact on image processing, delivering great performance. However, these models are typically hard to understand. This means, among other things, that there are no guarantees that the models will be stable, or respect the symmetries in the problem. For this reason, there has been a lot of interest in geometric deep learning (e.g. \cite{Sifre2014RigidClassification,Sifre2013RotationDiscrimination,Oyallon2015DeepClassification,Cohen2021EquivariantNetworks,Cohen2016GroupNetworks,Cohen2019GeneralSpaces,Sherry2024DesigningODEs,Veeling2018RotationPathology,Bekkers2018RotoAnalysis,Bekkers2019SplineGroups,Liu2021BundelLearning,Carlsson2025GeometryLearning,Cassel2025BundleNetworks,Bellaard2025EuclideanSpace,Gerken2022EquivarianceImages}). By integrating ideas from classical processing methods, geometric deep learning models can moreover improve parameter efficiency, since network capacity does not have to be wasted on enforcing e.g. symmetry preservation. 
PDE-based convolutional neural networks, developed by Smets et al. \cite{Smets2022PDEGCNNs}, are 
a class of geometric deep learning methods in which the layers of the network solve PDEs known from classical image analysis. The coefficients of these PDEs, which have clear geometric interpretations, can then be learned from data. 
We develop and test new PDE-G-CNNs based on the RDS PDE as a new contribution of this paper compared to our previous conference article.

Next, in Sec.~\ref{sec:related_work}, we go into detail on previous work in the aforementioned fields. In Sec.~\ref{sec:contributions} we list the contributions of our work. Finally, we summarise the structure of the article in Sec.~\ref{sec:outline}.

\subsection{Related Work}\label{sec:related_work}

\subsubsection{Smoothing and Shock Filtering}
Let us highlight some work on the combination of smoothing and shock filters, since this forms the basis for RDS filtering. 

In~\cite{SW23} Schaefer and Weickert proposed diffusion shock inpainting as a combination of coherence-enhancing shock filtering~\cite{weickert2003shockfilters} and homogeneous diffusion~\cite{Ii62,Ii63a}, and they extended this method to RDS inpainting in \cite{schaefer2024regularisedds}. With that they continued a line of work that explicitly, e.g.~\cite{AM94,GSZ02}, or implicitly, e.g.~\cite{PM90,Sa96}, combine smoothing and shock filtering, though with a different goal: image inpainting~\cite{BSCB00,MM98a}, the task of filling in gaps in images.

\subsubsection{Multi-Orientation Image Processing}
Many PDE-based image processing techniques have been extended to position-orientation spaces $\mathbb{M}_d \cong \mathbb{R}^d \times S^{d - 1}$.
Total Roto-Translational Variation (TR-TV) on $\M$ has been studied by Chambolle \& Pock \cite{chambolle2019tv} and Smets et al.~\cite{smets2021tvflow,SSVM2019}; see Pragliola et al.~\cite{CalatroniSIAMreview2023} for a review on the topic. There are interesting links between TV flow and elastica which have inspired further regularisation methods \cite{Liu2023ElasticaRegularization,chambolle2019tv}.
Mean Curvature (MC) flows on $\mathbb{M}_d$ have been employed for various purposes including denoising and inpainting by Citti et al. \cite{Citti2016mcf} ($d = 2$) and St.~Onge et al. \cite{stonge2019mcf} ($d = 3$). 
Diffusion PDEs on $\mathbb{M}_d$ have been employed in denoising, inpainting, and neurogeometry \cite{cittisartiJMIV,Prandi,franken2009cedos,petitotbook}.
Morphology PDEs on $\mathbb{M}_d$ have been effective in geometric deep learning \cite{smets2024thesis,Bellaard2023Analysis} and fiber enhancement \cite{DuitsJMIV2013}.
RDS filtering combines diffusion for denoising with morphological shock filtering to preserve edges like TR-TV \cite{chambolle2019tv} and MC flows \cite{smets2021tvflow}; by extending the RDS processing to $\M$ we include preservation of crossing/overlapping structures.

In \cite{Sherry2025DiffusionShock}, the authors extended RDS filtering to position-orientation space $\M$, combining the edge-preserving denoising capabilities of RDS filtering with the benefits of processing with PDEs on $\M$, such as preservation of crossings \cite{franken2009cedos,smets2021tvflow}. One limitation of the novel methods, however, is that they require manual parameter tuning. In this work we have integrated the RDS PDE into the PDE-based convolutional neural network framework to address this limitation.

\tikzset{
    pdelayer/.pic={
        \node[plain] (-more) {\rotatebox{90}{$\cdots$}};
        \node[sqr, above of=-more] (-topPDE) {PDE};
        \node[sqr, below of=-more] (-bottomPDE) {PDE};

        \node[plain, right of=-topPDE] (topaffine) {};
        \node[plain, right of=-more] (midaffine) {Affine};
        \node[plain, right of=-bottomPDE] (bottomaffine) {};
        \node[sqr, fit=(topaffine) (midaffine) (bottomaffine)] (-affine) {};

        \node[sqr,fit=(-more) (-topPDE) (-bottomPDE) (-affine)] (-PDE) {};

        \draw (-topPDE) -- (-affine);
        \draw (-more) -- (-affine);
        \draw (-bottomPDE) -- (-affine);
    }
}

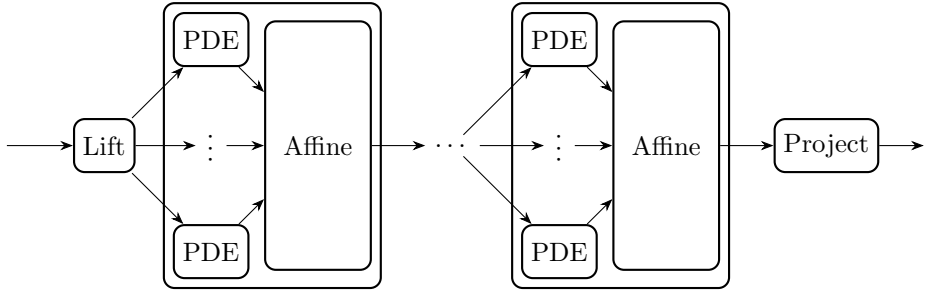
\begin{figure*}
\centering
\begin{tikzpicture}[
    node distance=4em,
    circ/.style={circle,draw,thick,minimum size=2em,align=center},
    sqr/.style={rectangle,draw,thick,rounded corners,minimum size=2em,align=center},
    plain/.style={draw=none,align=center},
    ->, >=Stealth
]
\node[plain] (U) {};
\node[sqr, right of=U] (lift) {Lift};
\pic[right of=lift] (leftPDE) {pdelayer};
\node[right=2em of leftPDE-affine] (morePDEs) {\dots};
\pic[right of=morePDEs] (rightPDE) {pdelayer};
\node[sqr, right=2em of rightPDE-affine] (project) {Project};
\node[plain, right of=project] (out) {};

\draw (U) -- (lift);

\draw (lift) -- (leftPDE-topPDE);
\draw (lift) -- (leftPDE-more);
\draw (lift) -- (leftPDE-bottomPDE);

\draw (leftPDE-affine) -- (morePDEs);

\draw (morePDEs) -- (rightPDE-topPDE);
\draw (morePDEs) -- (rightPDE-more);
\draw (morePDEs) -- (rightPDE-bottomPDE);

\draw (rightPDE-affine) -- (project);
\draw (project) -- (out);

\end{tikzpicture}
\caption{Schematic PDE-G-CNN. In PDE-CNNs, the `Lift' and `Project' layers can be omitted. In each layer, every feature map evolves according to a specified PDE with channel-dependent trainable parameters. Subsequently, the feature maps are mixed affinely, where the mixing weights are also trainable.}\label{fig:pde_g_cnn}
\end{figure*}

\subsubsection{PDE-Based Convolutional Neural Networks}\label{sec:pde_based_networks}
Smets et al. \cite{Smets2022PDEGCNNs} proposed a PDE-based generalisation of Cohen \& Welling's Group equivariant Convolutional Neural Networks (G-CNNs) \cite{Cohen2016GroupNetworks}. This framework can be applied to general Lie group homogeneous spaces; we consider only the following two types of PDE-based networks:
\begin{enumerate}
\item PDE-based Convolutional Neural Networks (PDE-CNNs) on $\Rtwo$ \cite{Bellaard2025PDECNNs}, and
\item PDE-based Group equivariant Convolutional Neural Networks (PDE-G-CNNs) on $\M$ \cite{Smets2022PDEGCNNs}.
\end{enumerate}
Fig.~\ref{fig:pde_g_cnn} schematically depicts a PDE-G-CNN. Note how it mirrors classical multi-orientation processing (cf. Fig.~\ref{fig:multiorientation_processing}): the data on $\Rtwo$ is first lifted, then processed in $\M$, and finally projected back to $\Rtwo$. In PDE-CNNs, the lifting and projection layers are omitted, and the PDE processing occurs in $\Rtwo$.

The layers in PDE-(G-)CNNs solve classical image processing PDEs: diffusion, dilation, and erosion \cite{Bellaard2025PDECNNs,Smets2022PDEGCNNs}. Since these PDEs generate a type of generalised (semifield) scale-space \cite{Bellaard2025PDECNNs}, they can be efficiently solved using generalised convolutions. For example, the diffusion PDEs are solved by linearly convolving with a kernel, while the dilation PDEs are solved by max pooling over an area defined by the kernel. 
These networks can also efficiently model convection PDEs. 

These networks are trained by learning the Riemannian metrics that parametrise the PDEs.  
The equivariance of PDE-(G-)CNNs depends on the invariance of the trained metrics. It is therefore possible to make roto-translation equivariant PDE-CNNs, but this restricts the metrics, and so also the kernels, to be isotropic. Conversely, since PDE-G-CNNs operate in $\M \cong \SE(2)$, the metrics do not need to be isotropic. Each PDE can therefore be more expressive without destroying equivariance. 
The metrics moreover have clear geometric interpretations. 
In particular, sub-Riemannian diagonal metrics can be used to approximate association fields, which model human perceptual grouping \cite{cittisartiJMIV,petitotbook,Bellaard2023Analysis}. Fig.~\ref{fig:exact_vs_log_ball} shows isosurfaces of a spatially anisotropic invariant Riemannian distance, and its computationally tractable logarithmic approximation \cite{Smets2022PDEGCNNs}. PDE-G-CNNs effectively learn the shapes of these isosurfaces.

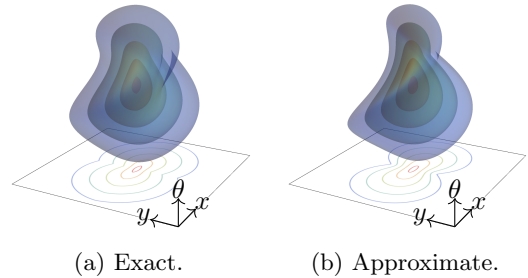
\begin{figure}
\centering
\begin{subfigure}[t]{0.22\textwidth}
\begin{tikzpicture}
\node[anchor=south west, inner sep=0] (image) at (0, 0) {\includegraphics[width=\textwidth]{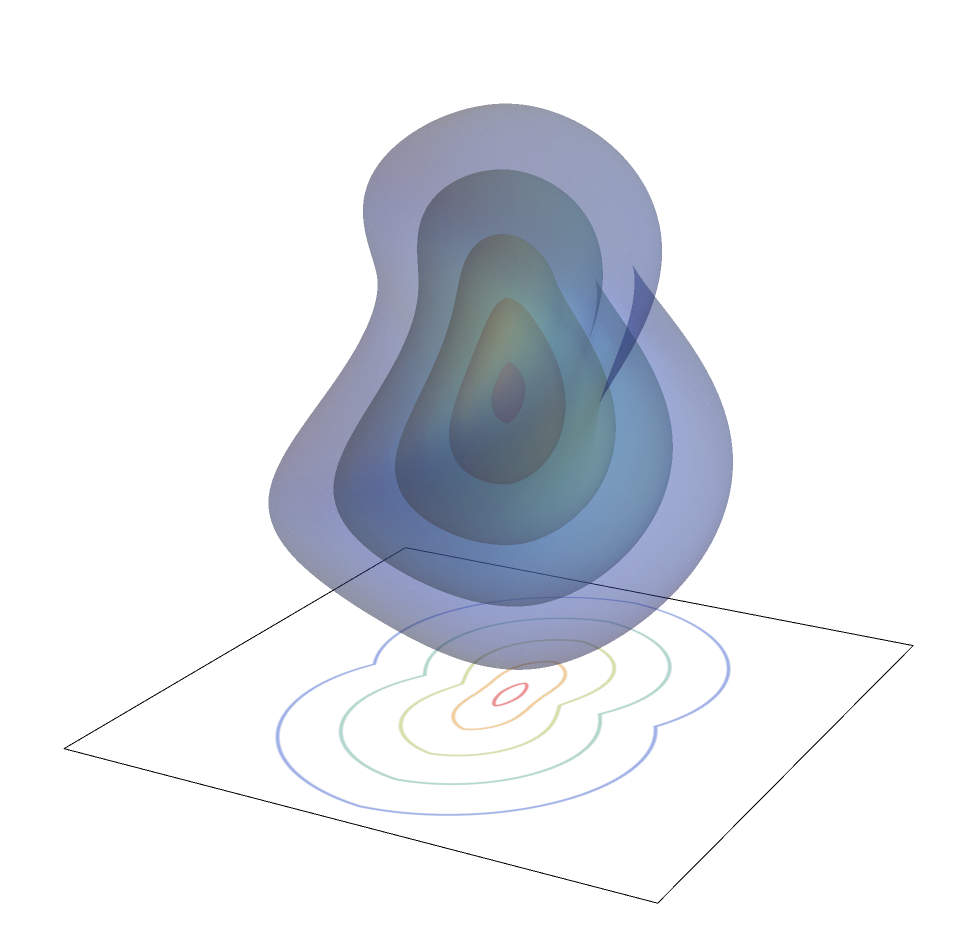}};
\begin{scope}[x={(image.south east)}, y={(image.north west)}, shift={(0.5, 0.5)}, scale=0.5]
\draw[->] (0.37, -0.945) -- (0.5, -0.805);
\draw(0.55, -0.76)node{\contour{white}{$x$}};
\draw[->] (0.37, -0.945) -- (0.16, -0.8885);
\draw(0.11, -0.87)node{\contour{white}{$y$}};
\draw[->] (0.37, -0.945) -- (0.37, -0.72);
\draw(0.37, -0.65)node{\contour{white}{$\theta$}};
\end{scope}
\end{tikzpicture}
\caption{Exact}\label{fig:exact_ball}
\end{subfigure}
\begin{subfigure}[t]{0.22\textwidth}
\begin{tikzpicture}
\node[anchor=south west, inner sep=0] (image) at (0, 0) {\includegraphics[width=\textwidth]{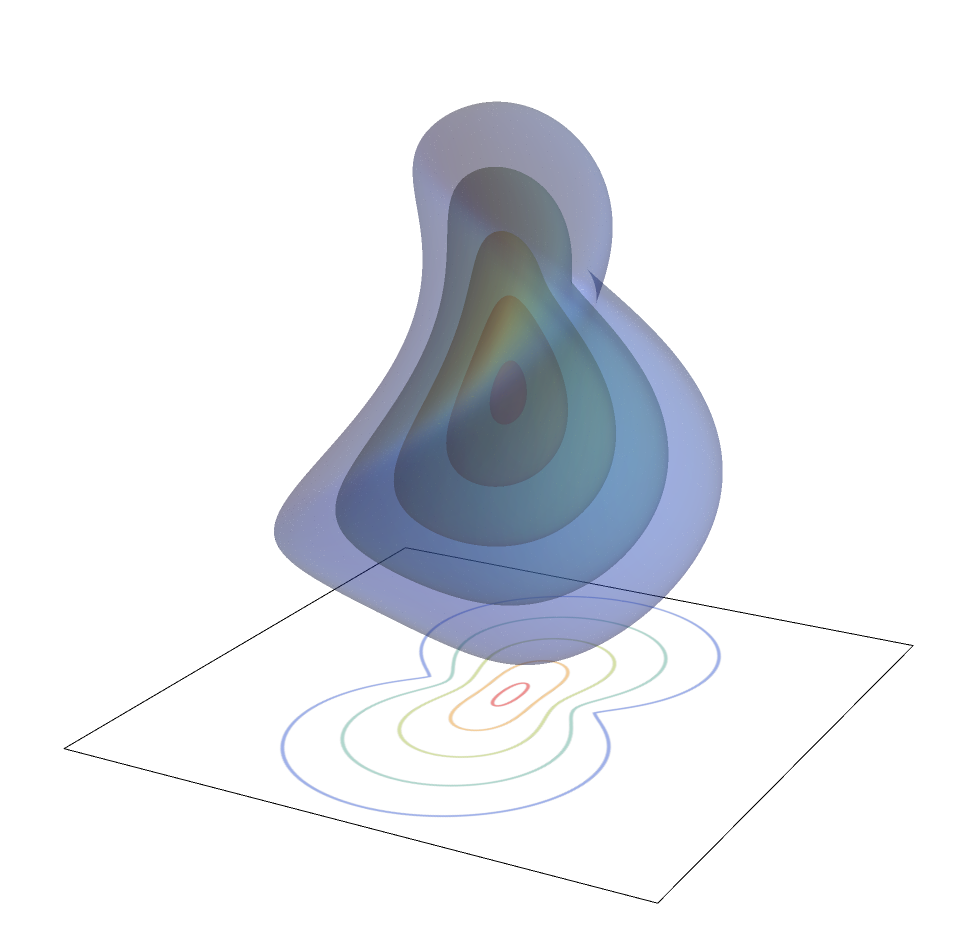}};
\begin{scope}[x={(image.south east)}, y={(image.north west)}, shift={(0.5, 0.5)}, scale=0.5]
\draw[->] (0.37, -0.945) -- (0.5, -0.805);
\draw(0.55, -0.76)node{\contour{white}{$x$}};
\draw[->] (0.37, -0.945) -- (0.16, -0.8885);
\draw(0.11, -0.87)node{\contour{white}{$y$}};
\draw[->] (0.37, -0.945) -- (0.37, -0.72);
\draw(0.37, -0.65)node{\contour{white}{$\theta$}};
\end{scope}
\end{tikzpicture}
\caption{Approximate}
\label{fig:log_ball}
\end{subfigure}
\caption{Isosurfaces of the exact Riemannian distance and its logarithmic aproximation on $\M$ with spatial anisotropy $\zeta = \frac{1}{3}$. In PDE-G-CNNs, the shape of these isosurfaces are learned.}\label{fig:exact_vs_log_ball}
\end{figure}

A nice property of PDE-(G-)CNNs is that they do not require untrainable non-linearities such as ReLU to be expressive due to the presence of trainable linear and non-linear PDE layers. There are other architectures that also include trainable non-linearities. In particular, the Soft Threshold Dynamics (STD) framework for segmentation replaces sigmoid linearities with variational non-linearities, allowing for the incorporation of spatial regularisation and shape priors \cite{Liu2022DeepSegmentation,Tai2024PottsNetworks}. Using multigrid-based schemes, it is then possible to create encoder-decoder architectures with such trainable non-linearities \cite{Tai2024PottsNetworks}.

In Sec.~\ref{sec:trained_diffusion_shock_filtering} we describe a trained RDS filtering layer for both PDE-CNNs and PDE-G-CNNs. We subsequently test these new RDS layers on a denoising and an inpainting task. Since previous work \cite{Smets2022PDEGCNNs,Bellaard2023Analysis,Bellaard2025PDECNNs} has shown that PDE-(G-)CNNs can outperform (G-)CNNs on segmentation tasks, while significantly reducing the number of trainable parameters and improving data efficiency, we use existing PDE layers as a baseline.

\subsection{Our Contributions}\label{sec:contributions}
As previously mentioned, this paper extends \cite{Sherry2025DiffusionShock}, which was presented at SSVM 2025. We develop invariant and gauge frame RDS filtering on position-orientation space $\M$. This involves defining a generalised notion of Laplacians, generated not by the Levi-Civita connection, but by Lie-Cartan connections (Def.~\ref{def:generalised_laplacian}). We theoretically analyse these generalised Laplacians: In the invariant setting, the Laplacian used for the diffusion part of the RDS filtering coincides with the Laplace-Beltrami operator (Thm.~\ref{thm:left_invariant_lie_cartan_laplacian}), while in the gauge frame setting, a discrepancy occurs (Thm.~\ref{thm:data_driven_lie_cartan_laplacian}), which has been previously overlooked in \cite{franken2009cedos,smets2021tvflow}. We assess the denoising and inpainting capabilities of our novel methods experimentally. Our denoising experiments show that our novel methods outperform other PDE-based methods (RDS filtering on $\Rtwo$ \cite{SW23,schaefer2024regularisedds}, TR-TV flows on $\M$ \cite{smets2021tvflow}), as well as non-PDE-based methods (NLM \cite{nlm}, BM3D \cite{bm3dold}) in terms of Peak Signal-to-Noise Ratio (PSNR) in denoising tasks. We moreover confirm that RDS filtering on $\M$ allows for inpainting crossing structures unlike $\Rtwo$ RDS filtering. The implementations of our novel methods and the experiments are available at \url{https://github.com/finnsherry/M2RDSFiltering}.

In this work we present a number of new contributions in addition to those from \cite{Sherry2025DiffusionShock}. We prove that the evolution generated by our generalised Laplacians is analytic (Thm.~\ref{thm:data_driven_lie_cartan_laplacian_sectorial}), which we conjectured in \cite{Sherry2025DiffusionShock}. We show that the $\M$ RDS filtering schemes satisfy a stability principle (Thm.~\ref{thm:stability}), in analogy to \cite[Thm.~1]{schaefer2024regularisedds}. We also develop geometric machine learning algorithms based on RDS filtering, by creating new RDS layers for PDE-(G-)CNNs. Unlike the PDEs previously used in PDE-(G-)CNNs, RDS does not generate a scale space. Therefore, the evolution generated by the RDS PDE cannot be computed simply by convolution, and instead requires a gating approach. We finally compare the new RDS PDE-(G-)CNNs to existing PDE-(G-)CNNS on denoising and inpainting tasks. The implementations of the PDE-(G-)CNNs and the experiments have been added to the open source Python package LieTorch \cite{Smets2022PDEGCNNs}, available at \url{https://gitlab.com/bsmetsjr/lietorch}. 

\subsection{Structure of the Article}\label{sec:outline}
We start by giving an overview of the necessary mathematical preliminaries in Sec.~\ref{sec:preliminaries}. In Sec.~\ref{sec:m2_ds}, we introduce RDS filtering on position-orientation space $\M$. For this, we first define and study generalised Laplacians, which are generated by Lie-Cartan connections as opposed to the Levi-Civita connection. In particular, we show how these generalised Laplacians differ from the Laplace-Beltrami operator (Thms.~\ref{thm:left_invariant_lie_cartan_laplacian},~\ref{thm:data_driven_lie_cartan_laplacian}), and that they still generate analytic semigroups (Thm.~\ref{thm:data_driven_lie_cartan_laplacian_sectorial}, Cor.~\ref{cor:lie_cartan_laplacian_sectorial}), which are diffusion-like (well-posed and smoothing). We subsequently compare RDS filtering to existing denoising methods, and show how on $\M$ we can inpaint crossing structures unlike in $\Rtwo$. We integrate RDS filtering into PDE-(G-)CNNs in Sec.~\ref{sec:trained_diffusion_shock_filtering}, and compare the new RDS PDE layers to existing PDE layers on a denoising and inpainting task. 

\section{Preliminaries}\label{sec:preliminaries}
We here summarise multi-orientation processing and RDS filtering on $\Rtwo$.

\subsection{Multi-Orientation Processing}\label{sec:m2}
Many multi-orientation image processing techniques have been developed and successfully applied to a large variety of tasks, including denoising \cite{chambolle2019tv,stonge2019mcf,smets2021tvflow}, regularisation \cite{CalatroniSIAMreview2023}, and line and contour enhancement \cite{Citti2016mcf,franken2009cedos,DuitsJMIV2013}. These methods work by lifting the data from Euclidean space to \emph{position-orientation space}, so that orientation information is encoded in the domain. In this way, crossing and overlapping structures can be disentangled, as shown in Fig.~\ref{fig:multiorientation_processing}. 
The construction of orientation-aware operators, which can be used to filter oriented features such as blood vessels \cite{Hannink2014CrossingVesselness}, is also simplified. 

Multi-orientation processing has previously been successfully applied to both two dimensional data (e.g. \cite{chambolle2019tv,stonge2019mcf,smets2021tvflow,Citti2016mcf,franken2009cedos,DuitsJMIV2013}) and three dimensional data (e.g. \cite{smets2021tvflow,Franken2011DiffusionsImages,Duits2011FiberMRI,Portegies2015ImprovingDeconvolution}). In this work, we only process images, which live on two dimensional Euclidean space $\Rtwo$, and hence we need to use the corresponding position-orientation space:
\begin{definition}[Position-orientation space\texorpdfstring{ $\M$}{}]\label{def:m2}
The \emph{position-orientation space} $\M$ is defined as the smooth manifold
\begin{equation}\label{eq:m2}
\M \coloneqq (T \Rtwo) \setminus \{0\} / \sim
\end{equation}
where $T$ denotes the tangent bundle and 0 denotes 0-section, and the equivalence relation $\sim$ is given, for $(\vec{x}_1, \dot{\vec{x}}_1), (\vec{x}_2, \dot{\vec{x}}_2) \in T \Rtwo \setminus \{0\}$, by
\begin{multline*}
(\vec{x}_1, \dot{\vec{x}}_1) \sim (\vec{x}_2, \dot{\vec{x}}_2) \iff \\
(\vec{x}_1 = \vec{x}_2 \textrm{ and } \exists \lambda > 0: \dot{\vec{x}}_1 = \lambda \dot{\vec{x}}_2).
\end{multline*}
It follows that $\M \cong \Rtwo \times S^1$, so that we may alternatively denote an element $(\vec{x}, \dot{\vec{x}}) \in \M$ by $(\vec{x}, \theta) \in \Rtwo \times \R / 2\pi \Z$, with $\theta$ the angle that $\dot{\vec{x}}$ makes with the positive $x$-axis.\footnote{This parametrisation is not a true coordinate chart, as there is a discontinuity at $\theta = 0$, but one can easily create two overlapping charts that cover $\M$ from this parametrisation.}
\end{definition}
Position-orientation space $\M$ is naturally acted on by the roto-translation group $\SE(2)$:
\begin{definition}[Special Euclidean Group \texorpdfstring{$\SE(2)$}{SE(2)}]\label{def:se2}
The \emph{2D special Euclidean group} $\SE(2)$ is defined as the Lie group of roto-translations acting on 2D Euclidean space, so $\SE(2) \coloneqq \Rtwo \rtimes \SO(2)$. The group product is given, for $(\vec{x}, R), (\vec{y}, S) \in \SE(2)$, by
\begin{equation}\label{eq:se2_product}
(\vec{x}, R) (\vec{y}, S) \coloneqq (\vec{x} + R \vec{y}, R S).
\end{equation}
The identity element is $e \coloneqq (0, I)$, and the inverse is $(\vec{x}, R)^{-1} \coloneqq (-R^{-1} \vec{x}, R^{-1})$ for $(\vec{x}, R) \in \SE(2)$. We write $R_\theta \in \SO(2)$ for a counter-clockwise rotation by $\theta \in \R / 2\pi \Z$.
\end{definition}
The action $L$ of $\SE(2)$ on $\M$ is given by: 
\begin{equation}\label{eq:action_se2_m2}
L_{(\vec{x}, R_\theta)} (\vec{y}, \phi) \coloneqq (R_\theta \vec{y} + \vec{x}, \phi + \theta),
\end{equation}
with $(\vec{x}, R_\theta) \in \SE(2)$ and $(\vec{y}, \phi) \in \M$.
By choosing reference element $p_0 \coloneqq (\vec{0}, 0) \in \M$, we see that \cite{smets2021tvflow}
\begin{equation}\label{eq:m2_principal_homogeneous_space}
\M \cong \SE(2) / \Stab_{\SE(2)}(p_0) \cong \SE(2),
\end{equation}
with stabiliser $\Stab_{\SE(2)}(p_0) \coloneqq \{g \in \SE(2) \mid L_g p_0 = p_0\}$, so that $\M$ is the \emph{principal homogeneous space} of $\SE(2)$.\footnote{Since $\M$ is diffeomorphic to $\SE(2)$, it is common to identify them. We will however distinguish between position-orientation space $\M$ and the roto-translation group $\SE(2)$ acting thereon.} As a consequence, $\M$ inherits many nice Lie group properties. In particular, it ensures the existence of a global frame of $\SE(2)$ \emph{invariant vector fields}, i.e. vector fields $\A \in \sections(T\M)$ such that 
\begin{equation}\label{eq:invariant_vector_field}
(L_g)_* \A_p = \A_{L_g p}, \textrm{ for } g \in \SE(2), p \in \M,
\end{equation}
where $\sections(T\M)$ are the smooth sections of the tangent bundle $T\M$, and $(L_\cdot)_*$ denotes the pushforward of the action $L_\cdot$. We denote the set of invariant vector fields by 
\begin{equation}\label{eq:left_invariant_vector_fields}
\mathfrak{X}(\M) \coloneqq \{\A \in \sections(T\M) \mid \A \textrm{ satisfies \eqref{eq:invariant_vector_field}}\}.
\end{equation}
Invariant vector fields are equivariant operators (see App.~\ref{sec:equivariance} for details on equivariance):
\begin{equation}\label{eq:invariant_vector_fields_equivariant}
\begin{split}
(\cL_g \after \A)|_p f & = \A_{L_g^{-1} p} f = (L_g^{-1})_* \A_p f \\
&  = \A_p (f \after L_g^{-1}) = (\A \after \cL_g)|_p f,
\end{split}
\end{equation}
with $g \in \SE(2)$, $p \in \M$, $\A \in \mathfrak{X}(\M)$, and $f \in \Ltwo(\M)$, making them suitable building blocks for our equivariant processing: we will use them in the generators of our equivariant diffusion-shock evolutions.  
We choose the following invariant frame:
\begin{definition}[Invariant Frame]\label{def:left_invariant_frame}
We define the invariant vector fields $\A_1, \A_2, \A_3 \in \mathfrak{X}(\M)$ as 
\begin{equation}\label{eq:left_invariant_vector_frame}
\begin{split}
\A_1|_p & \coloneqq (L_{g_p})_* \partial_x|_{p_0}, \\
\A_2|_p & \coloneqq (L_{g_p})_* \partial_y|_{p_0}, \textrm{ and } \\
\A_3|_p & \coloneqq (L_{g_p})_* \partial_\theta|_{p_0},
\end{split}
\end{equation}
where $g_p = (\mathbf{x}, R_\theta) \in \SE(2)$ for $p = (\mathbf{x}, \theta) \in \M$.
Together, they form a basis for $\mathfrak{X}(\M)$.
\end{definition}
With respect to the fixed coordinate frame $\{\partial_x, \partial_y, \partial_\theta\}$, the invariant frame is given, for $p = (\vec{x}, \theta) \in \M$, by $\A_1|_p = \cos(\theta) \partial_x|_p + \sin(\theta) \partial_y|_p$, $\A_2|_p = -\sin(\theta) \partial_x|_p + \cos(\theta) \partial_y|_p$, and $\A_3|_p = \partial_\theta|_p$.
From now, we will refer to our chosen frame $\{\A_i\}_i$ simply as \emph{the invariant frame}.
Working in the invariant frame has other upsides in addition to equivariance. In particular, $\A_1$ points spatially along the local orientation, while $\A_2$ points laterally. It is consequently easy to construct operators that detect lines and edges, for instance, which are necessary of RDS filtering (e.g. Eqs.~\eqref{eq:shock_R2} and \eqref{eq:morph_switch}).

\begin{figure*}
\centering
\begin{subfigure}[t]{0.2\textwidth}
\begin{tikzpicture}
\node[anchor=south west, inner sep=0] (image) at (0, 0) {\includegraphics[width=\textwidth]{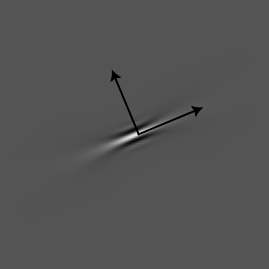}};
\begin{scope}[x={(image.south east)}, y={(image.north west)}, shift={(0.5, 0.5)}, scale=0.5]
\draw(0.59, 0.21)node{\contour{gray}{$\A_1$}};
\draw(-0.2, 0.56)node{\contour{gray}{$\A_2$}};
\end{scope}
\end{tikzpicture}
\caption{Cake wavelet}\label{fig:cakewavelet}
\end{subfigure}
\begin{subfigure}[t]{0.225\textwidth}
\begin{tikzpicture}
\node[anchor=south west, inner sep=0] (image) at (0, 0) {\includegraphics[width=\textwidth]{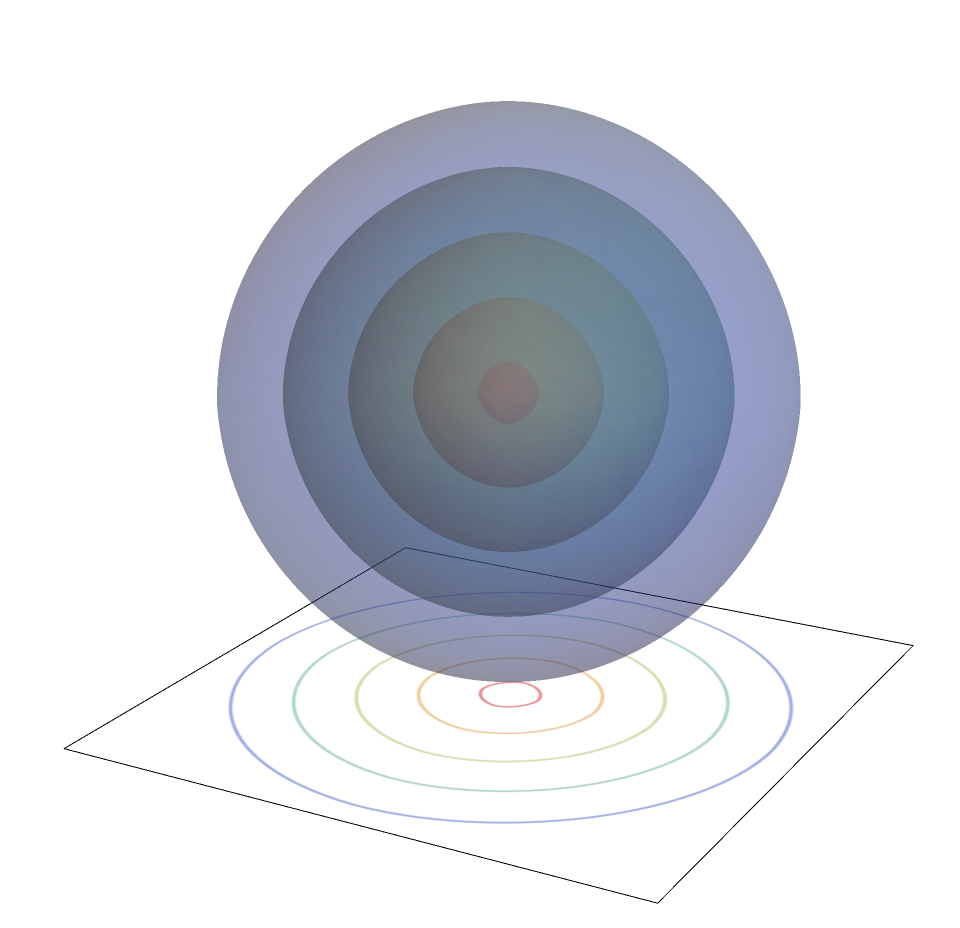}};
\begin{scope}[x={(image.south east)}, y={(image.north west)}, shift={(0.5, 0.5)}, scale=0.5]
\draw[->] (0.37, -0.945) -- (0.5, -0.805);
\draw(0.55, -0.76)node{\contour{white}{$x$}};
\draw[->] (0.37, -0.945) -- (0.16, -0.8885);
\draw(0.11, -0.87)node{\contour{white}{$y$}};
\draw[->] (0.37, -0.945) -- (0.37, -0.72);
\draw(0.37, -0.65)node{\contour{white}{$\theta$}};
\end{scope}
\end{tikzpicture}
\caption{$\zeta = 1$}\label{fig:exact_1_ball}
\end{subfigure}
\begin{subfigure}[t]{0.225\textwidth}
\begin{tikzpicture}
\node[anchor=south west, inner sep=0] (image) at (0, 0) {\includegraphics[width=\textwidth]{Figures/exact_3_ball.png}};
\begin{scope}[x={(image.south east)}, y={(image.north west)}, shift={(0.5, 0.5)}, scale=0.5]
\draw[->] (0.37, -0.945) -- (0.5, -0.805);
\draw(0.55, -0.76)node{\contour{white}{$x$}};
\draw[->] (0.37, -0.945) -- (0.16, -0.8885);
\draw(0.11, -0.87)node{\contour{white}{$y$}};
\draw[->] (0.37, -0.945) -- (0.37, -0.72);
\draw(0.37, -0.65)node{\contour{white}{$\theta$}};
\end{scope}
\end{tikzpicture}
\caption{$\zeta = \frac{1}{3}$}\label{fig:exact_3_ball}
\end{subfigure}
\begin{subfigure}[t]{0.225\textwidth}
\begin{tikzpicture}
\node[anchor=south west, inner sep=0] (image) at (0, 0) {\includegraphics[width=\textwidth]{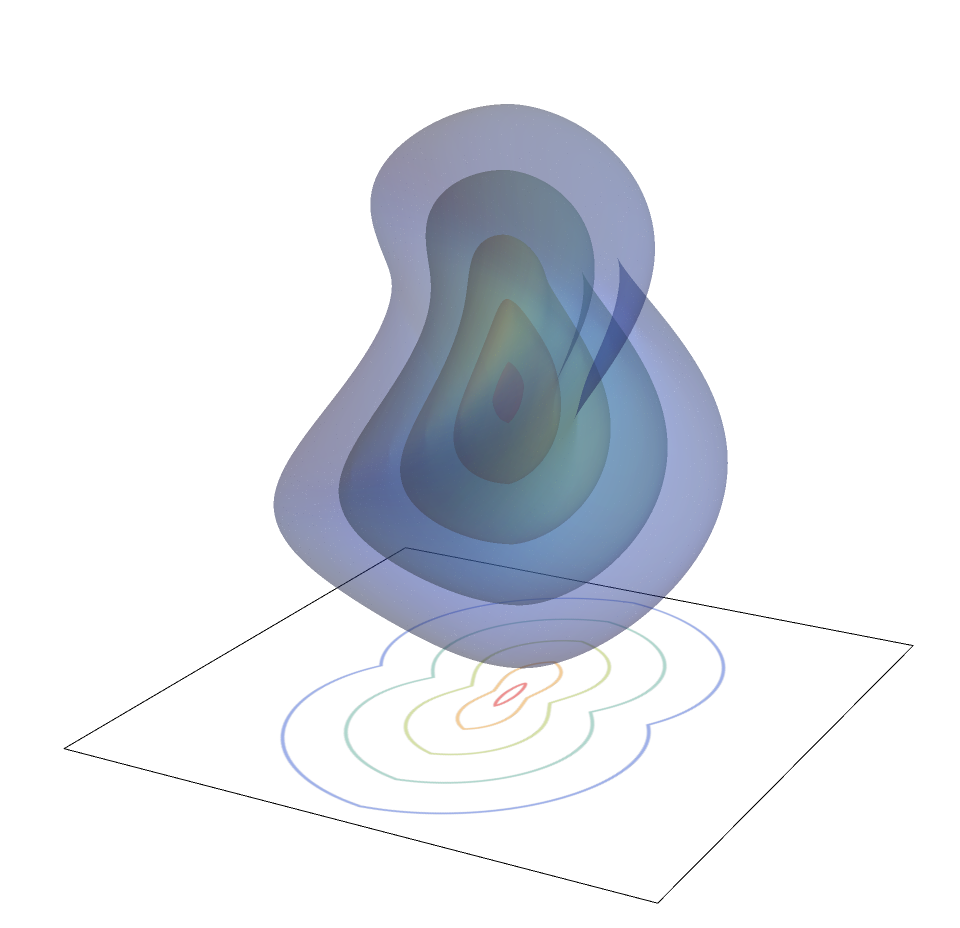}};
\begin{scope}[x={(image.south east)}, y={(image.north west)}, shift={(0.5, 0.5)}, scale=0.5]
\draw[->] (0.37, -0.945) -- (0.5, -0.805);
\draw(0.55, -0.76)node{\contour{white}{$x$}};
\draw[->] (0.37, -0.945) -- (0.16, -0.8885);
\draw(0.11, -0.87)node{\contour{white}{$y$}};
\draw[->] (0.37, -0.945) -- (0.37, -0.72);
\draw(0.37, -0.65)node{\contour{white}{$\theta$}};
\end{scope}
\end{tikzpicture}
\caption{$\zeta = \frac{1}{5}$}\label{fig:exact_5_ball}
\end{subfigure}
\caption{(a) Cake wavelet for orientation $\theta = \pi / 8$. (b-d) Isosurfaces of invariant Riemannian distances with varying spatial anisotropies $\zeta \coloneqq \sqrt{g_{11} / g_{22}}$.}
\label{fig:summary_M2}
\end{figure*}

We can furthermore equip $\M$ with a Riemannian metric tensor field (metric for short) $\G$. The metric defines an inner product on every tangent space, giving us a notion of lengths and angles, and moreover allows us to properly define differential operators like the gradient and Laplacian, as we will discuss in Sec.~\ref{sec:generalised_laplacians}. Finally, we can use it to define a distance $d_\G: \M \times \M \to \R_{\geq 0}$ via
\begin{equation}\label{eq:m2_distance}
d_\G(p, q) \coloneqq \inf_{\gamma \in \Gamma_p^q} \int_0^1 \sqrt{\G_{\gamma(t)}(\dot{\gamma}(t), \dot{\gamma}(t))} \diff t,
\end{equation}
where $\Gamma_p^q \coloneqq \{\gamma \in \operatorname{PC}^1([0, 1], \M) \mid \gamma(0) = p, \gamma(1) = q\}$. For equivariant processing we should choose invariant metrics, so that $(L_g)^* \G = \G$ for all $g \in \SE(2)$, with $(L_\cdot)^*$ the pullback of $L_\cdot$. It is not hard to see that a metric is invariant if and only if it is constant with respect to the invariant frame
\begin{equation}\label{eq:invariant_metrics}
\G(\A_i, \A_j) = g_{ij} \in \R, \textrm{ for all } i, j.
\end{equation}
Invariant metrics that are diagonal with respect to the invariant frame, i.e. $\G(\A_i, \A_j) = 0$ for $i \neq j$, have been studied frequently in diverse contexts in the past (e.g. \cite{chambolle2019tv,smets2021tvflow,franken2009cedos,cittisartiJMIV,Prandi,Bellaard2023Analysis,Hannink2014CrossingVesselness,Smets2022PDEGCNNs,bekkers2015subriemanniangeodesics,Duits2018OptimalAnalysis}) due to their transparent interpretation. In this case, the curve optimisation \eqref{eq:m2_distance} can be related to the Reeds-Shepp car \cite{Reeds1990OptimalBackwards}: $g_{11}$, $g_{22}$, and $g_{33}$ then define the costs of the car moving forward, moving sideways, and turning, respectively \cite{Duits2018OptimalAnalysis}. Of particular interest are sub-Riemannian metrics and their spatially anisotropic approximations, so that ``the car cannot move sideways''. This is a natural constraint to impose in problems involving oriented line-like structures, including line and contour completion \cite{cittisartiJMIV}, and perceptual grouping \cite{Berg2025ConnectedComponents}. Fig.~\ref{fig:exact_5_ball} shows the isosurfaces of such an anisotropic Riemannian distance; position-orientations are nearby in terms of this distance if they are simultaneously nearby in terms of position and orientation.

We can gain the benefits of multi-orientation processing for data on $\Rtwo$ by lifting the data to $\M$, using the \emph{orientation score transform}.
\begin{definition}[Orientation Score]\label{def:orientation_score}
The \emph{orientation score transform} $\W_\psi : \Ltwo(\Rtwo) \to \Ltwo(\M)$, where $\psi$ is a \emph{cake wavelet} (see Fig.~\ref{fig:cakewavelet}) \cite{Sherry2025OrientationCake}, is defined by
\begin{equation}\label{eq:ost}
\W_\psi f (\vec{x}, \theta) \coloneqq \int_{\Rtwo} \overline{\psi(R_{\theta}^{-1} (\vec{y} - \vec{x}))} f(\vec{y}) \diff \vec{y}
\end{equation}
for $f \in \Ltwo(\Rtwo)$ and $(\vec{x}, \theta) \in \M$. We then call $\W_\psi f$ the \emph{orientation score} of $f$.
\end{definition}
\begin{remark}
In practice, we use a finite number of orientations in a regular $N$-gon $S_N^1 \cong \{z = e^{i \theta} \in \mathbb{C} \mid z^N = 1\}$. 
\end{remark}
Crossings are disentangled by lifting to $\M$, opening the door for e.g. inpainting crossing structures, which is difficult in $\Rtwo$, see Fig.~\ref{fig:multiorientation_processing}.
Cake wavelets are designed such that data are lifted to the ``correct'' orientation in the sense that $\A_1$ points spatially along the data, while $\A_2$ points laterally \cite{franken2009cedos}. 
Moreover, it has been shown that they closely approximate $\SE(2)$ minimum uncertainty states \cite{Sherry2025OrientationCake}.
Finally, they also allow for fast approximate reconstruction using \cite{franken2009cedos}
\begin{equation}\label{eq:fast_reconstruction}
\begin{split}
f(\vec{x}) & = \W_\psi^{-1}(\W_\psi f) (\vec{x}) \approx \Proj(\W_\psi f) (\vec{x}) \\
&  \coloneqq \sum_{\theta \in S_N^1} \W_\psi f (\vec{x}, \theta).
\end{split}
\end{equation}
Hence, a typical multi-orientation image processing pipeline involves (1) lifting the data to position-orientation space $\M$ with orientation score transform $\W_\psi: \Ltwo(\Rtwo) \to \Ltwo(\M)$, (2) performing equivariant processing $\Phi: \Ltwo(\M) \to \Ltwo(\M)$ in $\M$, and (3) projecting back down to an image using $\Proj: \Ltwo(\M) \to \Ltwo(\Rtwo)$. In this work, $\Phi$ will be the solution operator of RDS filtering.

\subsubsection{Gauge Frames}\label{sec:gauge_frames}
Since we lift with a finite number of rotated cake wavelets, the orientation scores and vector fields are discretised in the orientational direction. Data therefore cannot always be lifted to exactly the correct orientation, and the vector field $\A_1$ only approximately points along the local orientation; the angle between the true spatial orientation and $\A_1$ is called \emph{deviation from horizontality} \cite{franken2009cedos}. 
Additionally, the direction of the lifted data will have an orientational component to account for the curvature in the input image, while $\A_1$ is purely spatial. 
Hence, it can be beneficial to use \emph{gauge frames}, which are adapted to the data.
\begin{definition}[First Gauge Vector]\label{def:first_gauge_vector}
Let $U \in C^2(\M)$, and let $\G$ be a metric on $\M$. The first gauge vector is given by \cite[Sec.~2.4]{smets2021tvflow}
\begin{equation}\label{eq:first_gauge_vector}
\A_1^U|_p \coloneqq \underset{\substack{X_p \in T_p \M \\ \norm{X_p}_{\G} = 1}}{\argmin} \norm*{\nabla_{X_p}^{[0]} \gradient_{\G} U}_{\G}^2,
\end{equation}
where $\nabla_\cdot^{[0]}$ is the $0$ Lie-Cartan connection, which will be discussed in Sec.~\ref{sec:generalised_laplacians}.
\end{definition}
We choose $\G \coloneqq \G_\xi$, with $\G_\xi(\A_1, \A_1) = \xi^2 = \G_\xi(\A_2, \A_2)$ and $\G_\xi(\A_3, \A_3) = 1$, and $\G_\xi(\A_i, \A_j) = 0$ for $i \neq j$. The parameter $\xi$, which can be interpreted as defining how a spatial unit relates to an orientational unit, has a large influence on the fitted gauge frame. The \enquote{correct} value of $\xi$, which depends e.g. on the number of orientations in the orientation score, can be determined by using the gauge frames to compute the curvature on images for which the ground truth curvature is known; we have found $\xi = 0.1$ to be appropriate when working with $32$ orientations.
\begin{figure}
\centering
\begin{subfigure}[t]{0.9\linewidth}
\begin{tikzpicture}
\node[anchor=south west, inner sep=0] (image) at (0, 0) {\includegraphics[width=\linewidth]{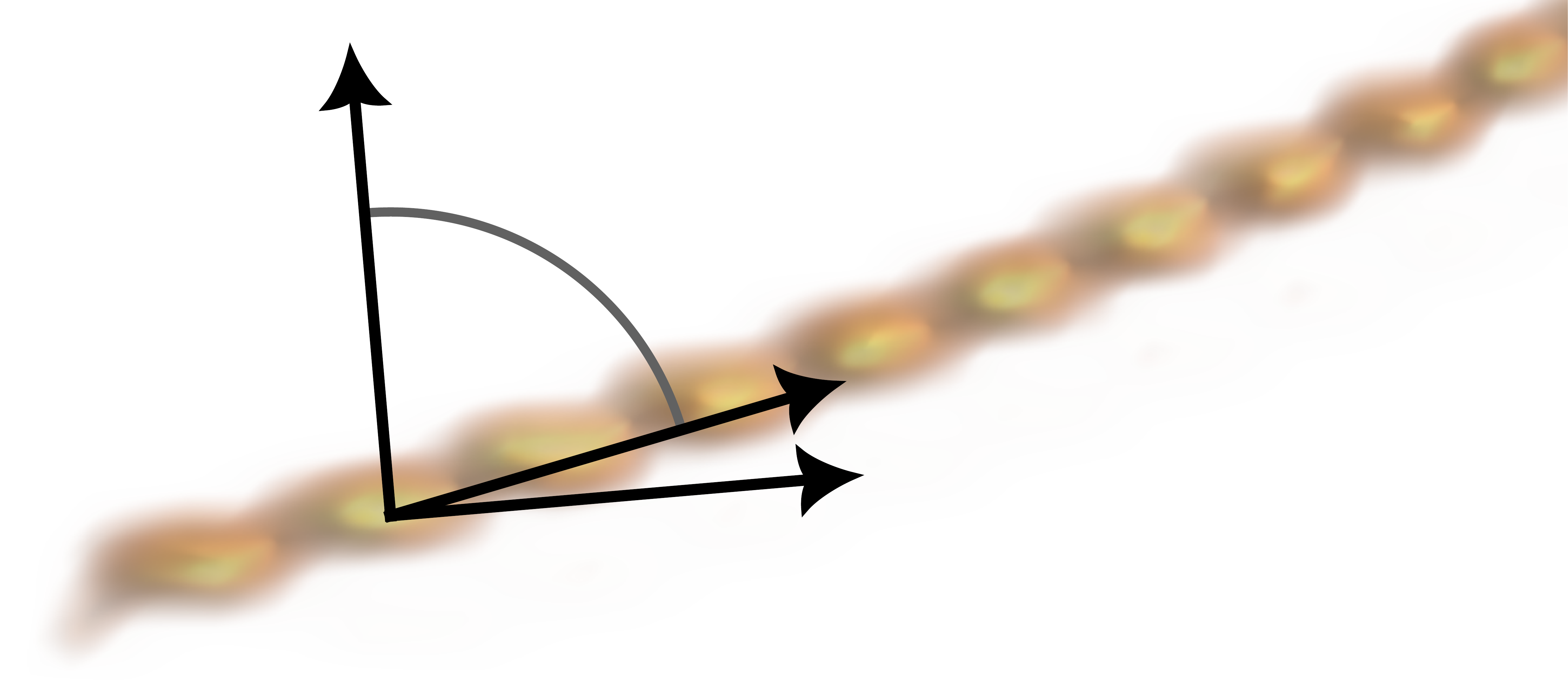}};
\begin{scope}[x={(image.south east)}, y={(image.north west)}, shift={(0.5, 0.5)}, scale=0.5]
\draw(0.2, -0.4)node{\contour{white}{$\A_1$}};
\draw(0.19, -0.08)node{\contour{white}{$\A_1^U$}};
\draw(-0.53, 0.95)node{\contour{white}{$\A_2$}};
\draw(-0.15, 0.47)node{\textcolor{blue}{$\frac{\pi}{2} - d_H$}};
\end{scope}
\end{tikzpicture}
\caption{Top view (along $\theta$-axis).}\label{fig:left_invariant_vs_gauge_top}
\end{subfigure}
~ 
\begin{subfigure}[t]{0.9\linewidth}
\begin{tikzpicture}
\node[anchor=south west, inner sep=0] (image) at (0, 0) {\includegraphics[width=\linewidth]{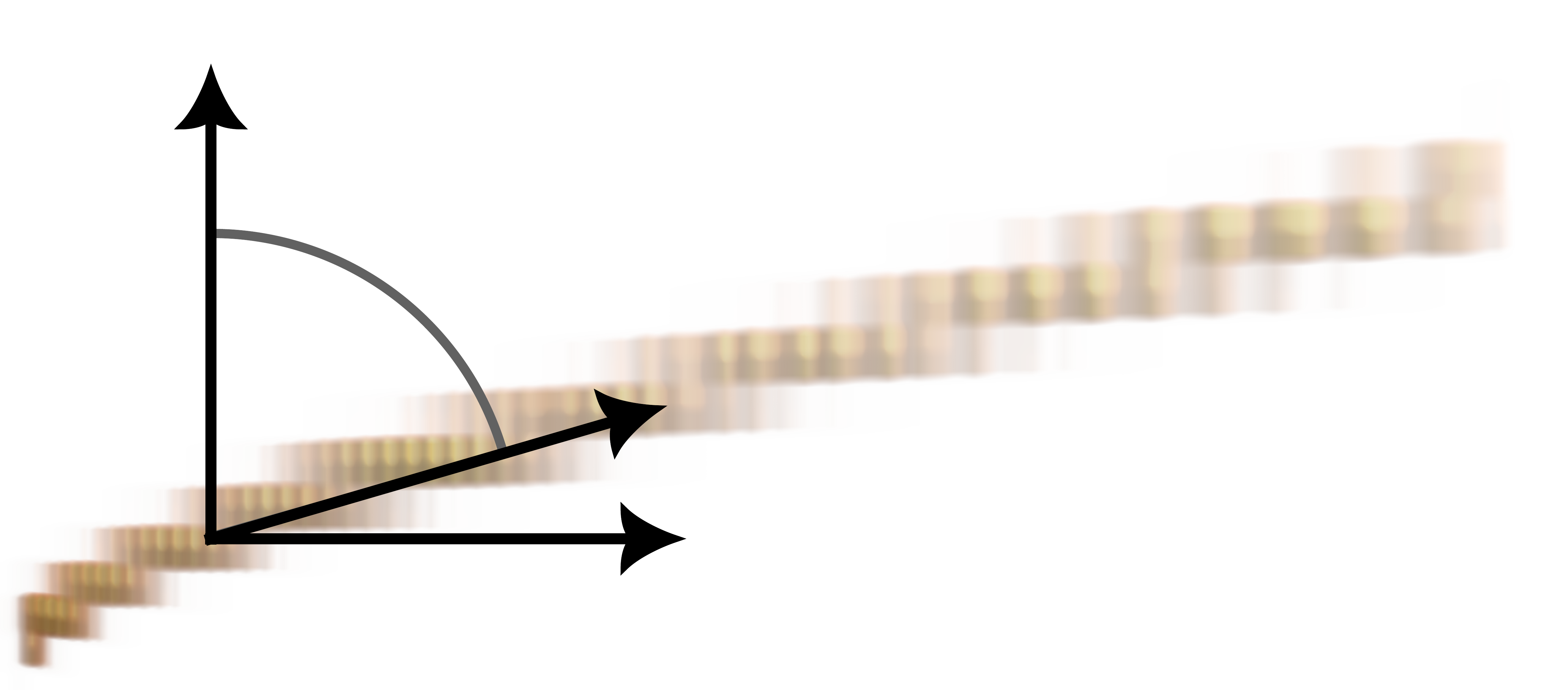}};
\begin{scope}[x={(image.south east)}, y={(image.north west)}, shift={(0.5, 0.5)}, scale=0.5]
\draw(-0.05, -0.6)node{\contour{white}{$\A_1$}};
\draw(-0.06, -0.18)node{\contour{white}{$\A_1^U$}};
\draw(-0.7, 0.95)node{\contour{white}{$\A_3$}};
\draw(-0.15, 0.35)node{\textcolor{blue}{$\frac{\pi}{2} - \arctan(\kappa)$}};
\end{scope}
\end{tikzpicture}
\caption{Side view (along $\A_2$-axis).}\label{fig:left_invariant_vs_gauge_side}
\end{subfigure}
\caption{Comparison of gauge frame and standard left-invariant frame. From the top view (a) we see that $\A_1^U$ has been rotated towards $\A_2$ to compensate for the deviation from horizontality $d_H$. From the side view (b) we see that $\A_1^U$ has been rotated towards $\A_3$; the rotation angle is related to the curvature $\kappa$.}
\label{fig:summary_gauge_frame}
\end{figure}
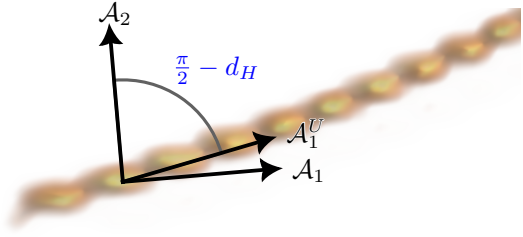
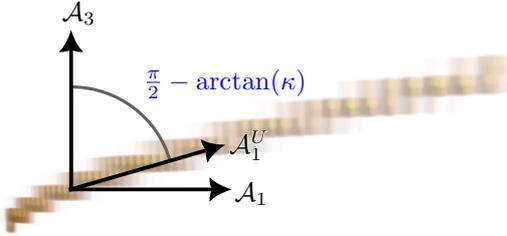

The other gauge vectors are then defined as follows: $\A_2^U$ is a purely spatial unit vector that is perpendicular to $\A_1^U$, and $\A_3^U$ is a unit vector perpendicular to both $\A_1^U$ and $\A_2^U$ with respect to $\G_\xi$, such that $\{\A_i^U\}_i$ is a right-handed frame. App.~\ref{app:computing_gauge_frame} describes how Eq.~\eqref{eq:first_gauge_vector} can be computed in practice to find the first gauge vector. 

Since the construction of the gauge frame \eqref{eq:first_gauge_vector} is equivariant (see Prop.~\ref{prop:equivariance_gauge_frame}), we may safely use it in our equivariant processing pipeline. If we define differential operators using a metric $\G^U$ that is constant with respect to the gauge frame, so
\begin{equation}\label{eq:data_driven_invariant_metrics}
\G^U(\A_i^U, \A_j^U) = g_{ij}^U \in \R, \textrm{ for all } i, j,
\end{equation}
then those operators will be equivariant too. For this reason, such metrics are also called \emph{data-driven invariant metrics}. We can have the same car-like intuition for diagonal data-driven invariant metrics as for invariant metrics.

Since fitting the gauge frame is computationally expensive, we only fit the gauge frame on the initial condition and this frame remains constant throughout the evolution. This is reasonable when enhancing an image as the structure of the orientation score, and consequently the gauge frame, does not change very much.

\subsection{Regularised Diffusion-Shock Filtering on \texorpdfstring{\unboldmath$\Rtwo$}{Euclidean Space}}\label{sec:r2_ds}
RDS filtering combines homogeneous diffusion with coherence-enhancing shock filtering. 
The coherence-enhancing shock filter sharpens and elongates edge-like structures by adaptively applying the morphological operations \emph{dilation} and \emph{erosion}~\cite{So04}. Dilation of a greyscale image $f: \Omega \subset \Rtwo \to \R$ replaces the grey value in location $\vec{x}$ by the supremum of $f$ within a specified neighbourhood around $\vec{x}$. Erosion uses the infimum instead. The PDE-based formulation~\cite{BM92,AGLM93,AVK93} of dilation $(+)$ / erosion ($-$) with a disk-shaped neighbourhood is given by
\begin{equation}
\partial_t u = \pm \abs{\gradient u},
\end{equation}
with the initial image $u(\vec{x}, 0) = f(\vec{x})$ and reflective boundary conditions at $\partial \Omega$, where $\abs{\,\cdot\,}$ and $\gradient$ are the Euclidean norm and gradient, respectively. 
The coherence-enhancing shock filter applies dilation when the data is concave in the direction perpendicular to the local orientation, and erosion when it is convex. 
This direction is determined by the dominant eigenvector $\vec{w}$ (i.e. the eigenvector with the largest eigenvalue) of a structure tensor~\cite{FG87} $\mat{J}_\rho(\gradient u_\sigma) = K_\rho * (\gradient u_\sigma \gradient u_\sigma^\top)$ where $u_\sigma = K_\sigma *u$ with the Gaussian convolution kernels $K_\rho, K_\sigma$. With that, the coherence-enhancing shock filter evolves an initial greyscale image $f$ by
\begin{equation}\label{eq:shock_R2}
\partial_t u = -S(\partial_{\vec{w}\vec{w}} u_\sigma) \abs{\gradient u},
\end{equation}
with initial condition $u(\vec{x}, 0) = f(\vec{x})$ and reflective boundaries.
The sigmoidal function $S: \R \to [-1, 1]$ implements the behaviour of a (soft) sign function. 

RDS filtering aims at applying this shock filter near edges, while the diffusion smooths flat areas. This adaptive behaviour is produced using a Charbonnier weight~\cite{charbonnier1997switch} $g: \R_{\geq 0} \to \R; x \mapsto \sqrt{1 + x/\lambda^2}^{-1}$ with a Gaussian-smoothed gradient magnitude $\gradient u_\nu$ as input. 
In summary, the RDS filtering PDE is given by
\begin{multline}\label{eq:ds_R2}
\partial_t u = g\left(\abs{\gradient u_\nu}^2\right) \laplace u \\
  - \Big(1 - g\left(\abs{\gradient u_\nu}^2\right)\Big)
  S\left(\partial_{\vec{w}\vec{w}} (u_\sigma) \right)
  \abs{\gradient u},
\end{multline}
with initial condition $u(\vec{x}, 0) = f(\vec{x})$ and reflective boundary conditions. 

For the application to digital images, the PDE can be discretised with an explicit scheme, that inherits a maximum-minimum principle which avoids under/over shoots. Diffusion is discretised with central differences, the morphological terms require an upwind scheme \cite{rouy1992viscosity}. For details see~\cite{SW23,schaefer2024regularisedds}. 

\section{Regularised Diffusion-Shock Filtering on \texorpdfstring{\unboldmath$\M$}{the Space of Positions \& Orientations}}\label{sec:m2_ds}
In this section, we discuss how to extend RDS filtering from $\Rtwo$ to $\M$. We start by investigating how to generalise diffusion. Thereafter, we describe our PDE scheme, and prove its stability. Finally, we test the denoising and inpainting capabilities of our novel scheme.

\subsection{Generalised Laplacians}\label{sec:generalised_laplacians}
\begin{remark}
We use Einstein summation convention in this section for concise  expressions. 
\end{remark}
On a Riemannian manifold $(M, \G)$, one often defines diffusion as the evolution generated by the Laplace-Beltrami operator $\laplace_\G \coloneqq \divergence_\G \after \gradient_\G$, where $\divergence_\G$ is the Riemannian divergence induced by the Riemannian volume form. It turns out that for any smooth vector field $X \in \sections(T M)$ it holds that $\divergence_\G(X) = \trace(\nabla_{\cdot}^\mathrm{LC} X)$, where $\nabla^\mathrm{LC}$ is the Levi-Civita connection corresponding to $\G$, see App.~\ref{app:divergence_connection}. This allows us to define a generalised notion of Laplace operator, where we replace the Levi-Civita connection with some other affine connection $\nabla$: 
\begin{definition}[Generalised Laplacian]\label{def:generalised_laplacian}
Let $(M, \G)$ be a Riemannian manifold, and $\nabla$ an affine connection thereon. Then we define the corresponding \emph{generalised Laplacian} as
\begin{equation}\label{eq:general_connection_laplacian}
\laplace_{\G, \nabla} \coloneqq \divergence_\nabla \after \gradient_\G \coloneqq \trace(\nabla_{\cdot} \gradient_\G).
\end{equation}
\end{definition}
These Laplace operators clearly generalise the Laplace-Beltrami operator, and could be interesting on manifolds which have natural connections that are \emph{not} the Levi-Civita connection.  In particular, on Lie groups there exists a family of canonical connections called the \emph{Lie-Cartan connections}.
\begin{definition}[Lie-Cartan Connection]\label{def:lie_cartan_connection}
Let $G$ be a Lie group and let $\nu \in \R$. Then the $\nu$ \emph{Lie-Cartan connection} $\nabla^{[\nu]}$ is the affine connection such that
\begin{equation}\label{eq:lie_cartan_connection}
\nabla^{[\nu]}_X Y = \nu [X, Y]
\end{equation}
for any invariant $X, Y \in \mathfrak{X}(G)$.
\end{definition}
\begin{remark}\label{rem:uniqueness_lie_cartan_connection}
The uniqueness of the above definition comes from the defining properties of affine connections, namely (1) $C^\infty(G)$-linearity in the first slot and (2) Leibniz rule in the second slot. Using these properties, we can express the $\nu$ Lie-Cartan connection with respect to an (arbitrary) invariant frame $\{\A_i\}_i$:
\begin{equation*}
\begin{split}
\nabla^{[\nu]}_X Y & = \nabla^{[\nu]}_{(X^i \A_i)} (Y^j \A_j) \overset{(1)}{=} X^i (\nabla^{[\nu]}_{\A_i} (Y^j \A_j)) \\
& \overset{(2)}{=} X^i (\A_i Y^j) \A_j + X^i Y^j \nabla^{[\nu]}_{\A_i} \A_j \\
& = X^i (\A_i Y^j) \A_j + \nu X^i Y^j [\A_i, \A_j],
\end{split}
\end{equation*}
for any (not necessarily invariant) $X, Y \in \sections(T G)$.
\end{remark}
Lie-Cartan connections have nice properties, such as the fact that their geodesics are exactly the exponential curves of $G$ \cite[Def.~2]{duits2021cartanconnection}. Of particular interest is $\nabla^{[0]}$, since this is the only one that is metric compatible with every invariant metric on $G$ \cite[Cor.~2]{duits2021cartanconnection}. We will now express the Laplace operators induced by Lie-Cartan connections and invariant metrics -- which we call \emph{Lie-Cartan Laplacians} -- and compare these with the Laplace-Beltrami operator. For readability, we write $\laplace_{\G, \nu} \coloneqq \laplace_{\G, \nabla^{[\nu]}}$. 
\begin{theorem}[Lie-Cartan Laplacians]\label{thm:left_invariant_lie_cartan_laplacian}
Let $G$ be a connected Lie group, let $\G$ be an invariant metric thereon, and let $\nu \in \R$. With respect to an invariant frame $\{\A_i\}_i$, the Lie-Cartan Laplacian is given by
\begin{equation}\label{eq:lie_cartan_laplacian_formula}
\laplace_{\G, \nu} = g^{ij} \A_i \A_j + \nu c_{ki}^k g^{ij} \A_j,
\end{equation}
while the Laplace-Beltrami operator is given by
\begin{equation}\label{eq:laplace_beltrami_formula}
\begin{split}
\laplace_\G & = g^{ij} \A_i \A_j + \Gamma_{ki}^k g^{ij} \A_j \\
& = g^{ij} \A_i \A_j + c_{ki}^k g^{ij} \A_j,
\end{split}
\end{equation}
with $\Gamma_{ij}^k$ the Christoffel symbols given by $\Gamma_{ij}^k \A_k = \nabla_{\A_i}^\mathrm{LC} \A_j$ and $c_{ij}^k$ the structure constants defined by $c_{ij}^k \A_k = [\A_i, \A_j]$. The difference is:
\begin{equation}\label{eq:difference_lie_cartan_laplacian_laplace_beltrami}
\laplace_\G - \laplace_{\G, \nu} = (1 - \nu) c_{ki}^k g^{ij} \A_j.
\end{equation}
\end{theorem}
\begin{proof}
We will start by deriving an expression for the divergences $\trace(\nabla_\cdot^{[\nu], U} X)$ and $\trace(\nabla_\cdot^\mathrm{LC} X)$ for all $X \in \sections(TG)$; the result is then achieved by substituting $X = \gradient_\G f$ for all $f \in C^\infty(G)$. We will express both sides of this equation in terms of the invariant frame $\{\A_i\}_i$. Let $\{\omega^i\}_i$ be the dual frame of $\{\A_i\}_i$. We first derive the divergence in the Lie-Cartan case:
\begin{equation*}
\begin{split}
\trace(\nabla_\cdot^{[\nu]} X) & \coloneqq \langle \omega^i, \nabla_{\A_i}^{[\nu]} (X^j \A_j) \rangle \\
& = \langle \omega^i, (\A_i X^j) \A_j + X^j \nabla_{\A_i}^{[\nu]} \A_j \rangle \\
& = \A_i X^i + \nu X^j \langle \omega^i, [\A_i, \A_j] \rangle \\
& = \A_i X^i + \nu X^j c_{ij}^j = (\A_i + \nu c_{ji}^j) X^i,
\end{split}
\end{equation*}
from which it follows that
\begin{equation*}
\laplace_{\G, \nu} = (\A_i + \nu c_{ji}^j) g^{ik} \A_k = g^{ij} \A_i \A_j + \nu c_{ki}^k g^{ij} \A_j.
\end{equation*}
Next, we derive in the Levi-Civita case:
\begin{equation*}
\begin{split}
\trace(\nabla_\cdot^\mathrm{LC} X) & \coloneqq \langle \omega^i, \nabla_{\A_i}^\mathrm{LC} (X^j \A_j) \rangle \\
& = \langle \omega^i, (\A_i X^j) \A_j + X^j \nabla_{\A_i}^\mathrm{LC} \A_j \rangle \\
& = \A_i X^i + \langle \omega^i, X^j \Gamma_{ij}^k \A_k \rangle \\
& = \A_i X^i + \Gamma_{ij}^i X^j = (\A_i + \Gamma_{ji}^j) X^i,
\end{split}
\end{equation*}
where $\Gamma_{ij}^k \coloneqq \langle \omega^k, \nabla_{\A_i}^\mathrm{LC} \A_j \rangle$ are the Christoffel symbols with respect to $\{\A_i\}_i$, which are given by \cite[Cor~5.11]{lee2018riemannian}
\begin{multline*}
\Gamma_{ij}^k = \frac{1}{2} g^{km} (\A_i g_{mj} + \A_j g_{mi} - \A_m g_{ij} \\ - g_{ip} c_{jm}^p - g_{jp} c_{im}^p + g_{mp} c_{ij}^p).
\end{multline*}
Since the metric is invariant, the derivatives of its components vanish, whence
\begin{equation*}
\Gamma_{ij}^k = -\frac{1}{2} g^{km} (g_{ip} c_{jm}^p + g_{jp} c_{im}^p - g_{mp} c_{ij}^p)
\end{equation*}
The relevant components are hence
\begin{equation*}
\begin{split}
\Gamma_{ji}^j & = -\frac{1}{2} g^{jm} (g_{jp} c_{im}^p + g_{ip} c_{jm}^p - g_{mp} c_{ji}^p) \\
& = -\frac{1}{2} (\delta_p^m c_{im}^p + g_{ip} g^{jm} c_{jm}^p - g_p^j c_{ji}^p) \\
& = \frac{1}{2} (c_{mi}^m + c_{ji}^j - g_{ip} g^{jm} c_{jm}^p).
\end{split}
\end{equation*}
Using the symmetry of the components of the (dual) metric tensor and the antisymmetry of the structure constants we find that
\begin{equation*}
\Gamma_{ji}^j = \frac{1}{2} (c_{mi}^m + c_{ji}^j - g_{ip} g^{jm} c_{jm}^p) = \frac{1}{2} (c_{mi}^m + c_{ji}^j) = c_{ji}^j.
\end{equation*}
We have therefore found that
\begin{equation*}
\trace(\nabla_\cdot^\mathrm{LC} X) = (\A_i + c_{ji}^j) X^i.
\end{equation*}
Substituting in $\gradient_\G$, we find
\begin{equation*}
\laplace_\G = (\A_i + c_{ji}^j) g^{ik} \A_k = g^{ij} \A_i \A_j + c_{ki}^k g^{ij} \A_j.
\end{equation*}
Thereby the difference of the Laplace-Beltrami operator and the Lie-Cartan Laplacian is
\begin{equation*}
\laplace_\G - \laplace_{\G, \nu} = (1 - \nu) c_{ki}^k g^{ij} \A_j,
\end{equation*}
as required.
\end{proof}
Equation~\eqref{eq:difference_lie_cartan_laplacian_laplace_beltrami} tells us that the Laplace-Beltrami operator $\laplace_\G$ and the $1$ Lie-Cartan Laplacian $\laplace_{\G, 1}$ coincide. The operators $\laplace_\G$ and $\laplace_{\G, \nu}$ coincide for all $\nu \in \R$ if the trace of the structure constants $c_{ki}^k$ vanishes, which happens if and only if $G$ is unimodular \cite{Varadarajan1984LieRepresentations}.
Hence, for unimodular Lie groups we can simplify Equations~\eqref{eq:lie_cartan_laplacian_formula} and \eqref{eq:laplace_beltrami_formula}:
\begin{corollary}\label{cor:left_invariant_lie_cartan_laplacian}
Let $G$ be a connected unimodular Lie group and $\G$ a left-invariant metric thereon, and let $\nu \in \R$. Then the Lie-Cartan Laplacian and Laplace-Beltrami operator coincide:
\begin{equation}\label{eq:unimodular_difference_lie_cartan_laplacian_laplace_beltrami}
\laplace_{\G, \nu} = g^{ij} \A_i \A_j = \laplace_\G.
\end{equation}
\end{corollary}
Many Lie groups -- including $\SE(2)$ and $\Rtwo$ -- are unimodular.\footnote{On $\Rtwo$, all structure constants vanish; on $\SE(2)$, the nonzero structure constants are $-1 = c_{13}^2 = -c_{31}^2 = - c_{23}^1 = c_{32}^1$, none of which contribute to the trace of the adjoint.} As such, Cor.~\ref{cor:left_invariant_lie_cartan_laplacian} carries over onto their principal homogeneous spaces $\M$ and $\Rtwo$. 

The theory of Lie-Cartan connections can be generalised for gauge frames:
\begin{definition}[Gauge Frame Lie-Cartan Connection]\label{def:data_driven_lie_cartan_connection}
Let $G$ be a Lie group, with gauge frame $\{\A_i^U\}_i$, and let $\nu \in \R$. Then the $\nu$ \emph{gauge frame Lie-Cartan connection} $\nabla^{[\nu], U}$ is the affine connection such that
\begin{equation}\label{eq:data_driven_lie_cartan_connection}
\nabla^{[\nu], U}_X Y = \nu [X, Y]
\end{equation}
for any $X, Y \in \sections(T G)$ that are constant with respect to $\{\A_i^U\}_i$.
\end{definition}
We call the corresponding Laplacians \emph{data-driven Lie-Cartan Laplacians}, and write $\laplace_{\G, \nu}^U \coloneqq \laplace_{\G, \nabla^{[\nu], U}}$.
\begin{theorem}[Data-Driven Lie-Cartan Laplacians]\label{thm:data_driven_lie_cartan_laplacian}
Let $G$ be a Lie group, with gauge frame $\{\A_i^U\}_i$, let $\G$ be a data-driven invariant metric thereon, and let $\nu \in \R$. With respect to the gauge frame, the data-driven Lie-Cartan Laplacian is given by
\begin{equation}\label{eq:data_driven_lie_cartan_laplacian_formula}
\laplace_{\G, \nu}^U = g^{ij} \A_i^U \A_j^U + \nu d_{ki}^k g^{ij} \A_j^U,
\end{equation}
while the Laplace-Beltrami operator is given by
\begin{equation}\label{eq:data_driven_laplace_beltrami_formula}
\begin{split}
\laplace_\G & = g^{ij} \A_i^U \A_j^U + \Gamma_{ki}^k g^{ij} \A_j^U \\
& = g^{ij} \A_i^U \A_j^U + d_{ki}^k g^{ij} \A_j^U,
\end{split}
\end{equation}
with $d_{ij}^k $ the structure functions defined by $d_{ij}^k \A_k^U = [\A_i^U, \A_j^U]$. The difference is: 
\begin{equation}\label{eq:difference_data_driven_lie_cartan_laplacian_laplace_beltrami}
\laplace_\G - \laplace_{\G, \nu}^U = (1 - \nu) d_{ki}^k g^{ij} \A_j^U,
\end{equation}
so the two coincide if and only if $\nu = 1$.\footnote{In the gauge frame case, structure \emph{constants} $c_{ij}^k$ are replaced by structure \emph{functions} $d_{ij}^k$ which do not have a closed form and have non-vanishing trace in general.}
\end{theorem}
\begin{proof}
The proof is analogous to that of Theorem~\ref{thm:left_invariant_lie_cartan_laplacian}: simply substitute $\A_i \to \A_i^U$, $\G \to \G^U$, and $c^k_{ij} \to d^k_{ij}$, and the result follows.
\end{proof}

Hence, for most $\nu \in \R$, the Laplace-Beltrami operator and the $\nu$ data-driven Lie-Cartan Laplacian will differ. We choose to use the $0$ data-driven Lie-Cartan Laplacian: 
\begin{equation*}
\begin{split}
\laplace & \coloneqq g^{11} (\A_1^U)^2 + g^{22} (\A_2^U)^2 + g^{33} (\A_3^U)^2 \\
& \equiv \laplace_{\G, 0}^U \neq \laplace_\G,
\end{split}
\end{equation*}
to generate the diffusion for gauge frame RDS filtering. One reason for this choice is that $\nabla^{[0], U}$ is the only gauge frame Lie-Cartan connection that is metric compatible with any data-driven invariant metric (App.~\ref{app:properties_lie_cartan}). 
It moreover has a number of computational advantages compared to the Laplace-Beltrami operator: (1) it does not depend on the structure functions, and (2) all derivatives are of second order, whereas the Laplace-Beltrami operator has both second and first order derivatives. 
This data-driven Laplacian has been used in the past, e.g. \cite{franken2009cedos}. Similarly, the gauge frame TR-TV flow from \cite{smets2021tvflow} implicitly uses a divergence induced by $\nabla^{[0], U}$. 

One expects diffusion to be well-behaved, e.g. smoothing. In fact, diffusions generated by the Laplace-Beltrami operator are even analytic; equivalently, we say that the Laplace-Beltrami operator is sectorial:
\begin{definition}[Sectorial Operators]\label{def:sectoriality}
Operator $(A, D(A))$ is called \emph{sectorial} if $A$ generates an analytic semigroup \cite[Thm.~II.4.6]{engel2000analyticsemigroup}.
\end{definition}
For completeness, we provide a short proof of the sectoriality of the Laplace-Beltrami operator.
\begin{remark}
In the following results, we will be working with function spaces defined with respect to the measure $\mu_\G$ induced by the Riemannian metric, which will differ from the Haar measure when $\G$ is not invariant. We simplify notation by omitting the measure in the function space notation (e.g. $\Ltwo(G, \mu_\G) \to \Ltwo(G)$). The sectoriality can be carried over to the Haar measure case if desired, since the two measures are absolutely continuous with respect to each other.\footnote{We relate the operator with the Riemannian measure $\mu_\G$ to the one with the Haar measure $\mu$ using unitary $\sqrt{\diff \mu_\G / \diff \mu}$, which preserves sectoriality.} Additionally, we use $H^2(G)$ to denote the Sobolev space of functions on $G$ with square integrable weak derivatives to second order \cite{Hebey2006SobolevManifolds}, which will turn out to be the domain of the Laplacians.
\end{remark}
\begin{lemma}[Laplace-Beltrami Operator Sectoriality]\label{lem:lbo_sectorial}
Let $G$ be a geodesically complete manifold with respect to Riemannian metric $\G$. Then the Laplace-Beltrami operator $(\laplace_\G, H^2(G))$ is sectorial.
\end{lemma}
\begin{proof}
By Green's identities, the Laplace-Beltrami operator is symmetric and negative on $C_c^\infty(G)$, since
\begin{align*}
(\laplace_\G f, g) & = \int_G \laplace_\G f(p) g(p) \diff \mu_\G(p) \\ 
& = -\int_G \G(\gradient_\G f(p), \gradient_\G g(p)) \diff \mu_\G(p) \\
& = \int_G f(p) \laplace_\G g(p) \diff \mu_\G(p) = (f, \laplace_\G g), \textrm{ and} \\
(\laplace_\G f, f) & = \int_G \laplace_\G f(p) f(p) \diff \mu_\G(p) \\
& = -\int_G \norm{\gradient_\G f(p)}_\G^2 \diff \mu_\G(p) \leq 0,
\end{align*}
for any $f, g \in C_c^\infty(G)$. Since $G$ is geodesically complete, it follows that $(\laplace_\G, C_c^\infty(G))$ is essentially self-adjoint \cite[Thm~2.2]{Chernoff1973EssentialEquations}, so that its closure $(\laplace_\G, H^2(G))$ is self-adjoint. Additionally, since negativity implies dissipativity, $(\laplace_\G, H^2(G))$ is dissipative \cite[Prop.~VI.3.14]{engel2000analyticsemigroup}. 
Consequently, $(\laplace_\G, H^2(G))$ is normal, and we have for its spectrum $\sigma(\laplace_\G, H^2(G)) \subset (-\infty, 0]$; it then follows that the Laplace-Beltrami operator is sectorial \cite[Cor.~III.4.7]{engel2000analyticsemigroup}.
\end{proof}
Hence, these evolutions are smoothing, so for all $t > 0$ and $f \in \Ltwo(G)$ we have \cite[Thm.~II.4.6]{engel2000analyticsemigroup}: 
\begin{equation}
e^{t \laplace_\G} f \in D(\laplace_\G^\infty) \subset C^\infty(G).
\end{equation}
The following result tells us that the data-driven Lie-Cartan Laplacians are also sectorial, so that the evolution they generate is indeed like a diffusion (well-posed and smoothing).
\begin{theorem}[Sectoriality of Data-Driven Lie-Cartan Laplacians]\label{thm:data_driven_lie_cartan_laplacian_sectorial}
Let $G$ be a Lie group, with gauge frame $\{\A_i^U\}_i$, let $\G$ be a data-driven invariant metric thereon, and let $\nu \in \R$. Assume that $(G, \G)$ is geodesically complete, and that the structure functions are bounded as follows: there is a $C > 0$ such that $g^{ij} d_{ki}^k(p) d_{lj}^l(p) \leq C$ for all $p \in G$.
Then the $\nu$ data-driven Lie-Cartan Laplacian $(\laplace_{\G, \nu}^U, H^2(G))$ is sectorial.
\end{theorem}
\begin{proof}
Thm.~\ref{thm:data_driven_lie_cartan_laplacian} tells us that
\begin{equation*}
\laplace_{\G, \nu}^U = \laplace_\G + (\nu - 1) d_{ki}^k g^{ij} \A_j^U,
\end{equation*}
where $\laplace_\G$ is sectorial by Lem.~\ref{lem:lbo_sectorial}. Hence, $\laplace_{\G, \nu}^U$ is a perturbation of a sectorial operator. For $\nu = 1$, the perturbation is trivial, so consider $\nu \neq 1$. We will now show that the perturbation $(\nu - 1) d_{ki}^k g^{ij} \A_j^U$ has a $\laplace_\G$-bound of 0, so that we may conclude by \cite[Thm.~III.2.10]{engel2000analyticsemigroup} that $\laplace_{\G, \nu}^U$ is sectorial too. The $\laplace_\G$-bound of $(\nu - 1) d_{ki}^k g^{ij} \A_j^U$ is in our case given by \cite[Def.~III.2.1]{engel2000analyticsemigroup}:
\begin{multline*}
a_0 \coloneqq \inf_{a > 0} \{\exists b \geq 0 \mid \forall f \in H^2(G): \\ \norm{(\nu - 1) d_{ki}^k g^{ij} \A_j^U f}_{\Ltwo(G)} \\ \leq a \norm{\laplace_\G f}_{\Ltwo(G)} + b \norm{f}_{\Ltwo(G)}\}.
\end{multline*}
This bound is 0 (so that the generator is sectorial) if and only if for all $a > 0$ there exists a $b \geq 0$ such that
\begin{equation*}
\norm{d_{ki}^k g^{ij} \A_j^U f}_{\Ltwo(G)} \leq a \norm{\laplace_\G f}_{\Ltwo(G)} + b \norm{f}_{\Ltwo(G)}
\end{equation*}
for all $f \in H^2(G)$, as we can simply divide out the $(\nu - 1)$ on both sides. Hence, let $a > 0$ and $f \in H^2(G)$, and consider that 
\begin{equation*}
\begin{split}
& \norm{d_{ki}^k g^{ij} \A_j^U f}_{\Ltwo(G)}^2 = \int_G \abs{d_{ki}^k(p) g^{ij} \A_j^U|_p f}^2 \diff \mu_\G(p) \\
& = \int_G \abs{\G_p(\gradient_\G f(p), d_{ki}^k(p) g^{ij} \A_j^U|_p)}^2 \diff \mu_\G(p) \\
& \leq \int_G \norm{\gradient_\G f(p)}_\G^2 \norm{d_{ki}^k(p) g^{ij} \A_j^U|_p}_\G^2 \diff \mu_\G(p).
\end{split}
\end{equation*}
Using the bound
\begin{equation*}
\begin{split}
\norm{d_{ki}^k(p) g^{ij} \A_j^U|_p}_\G^2 & = g_{mn} d_{ki}^k(p) g^{im} d_{ln}^l(p) g^{jn} \\
& = g^{ij} d_{ki}^k(p) d_{lj}^l(p) \leq C,
\end{split}
\end{equation*}
we then see that
\begin{equation*}
\begin{split}
& \norm{d_{ki}^k g^{ij} \A_j^U f}_{\Ltwo(G)}^2 \\
& \leq \int_G \norm{\gradient_\G f(p)}_\G^2 \norm{d_{ki}^k(p) g^{ij} \A_j^U|_p}_\G^2 \diff \mu_\G(p) \\
& \leq C \int_G \norm{\gradient_\G f(p)}_\G^2 \diff \mu_\G(p) = C \norm{\gradient_\G f}_{\Ltwo(G)}^2.
\end{split}
\end{equation*}
Next, we use Green's identities and the Cauchy-Schwarz inequality to see that
\begin{equation*}
\begin{split}
\norm{\gradient_\G f}_{\Ltwo(G)} & = \sqrt{\int_G \G_p(\gradient_\G f(p), \gradient_\G f(p)) \diff \mu_\G(p)} \\
& = \sqrt{\int_G -\laplace_\G f(p) f(p) \diff \mu_\G(p)} \\
& \leq \sqrt{\norm{\laplace_\G f}_{\Ltwo(G)} \norm{f}_{\Ltwo(G)}}.
\end{split}
\end{equation*}
We can then apply Young's inequality to find that for all $\epsilon > 0$
\begin{multline*}
\sqrt{\norm{\laplace_\G f}_{\Ltwo(G)} \norm{f}_{\Ltwo(G)}} \\ \leq \epsilon \norm{\laplace_\G f}_{\Ltwo(G)} + \frac{1}{4 \epsilon} \norm{f}_{\Ltwo(G)}.
\end{multline*}
Hence, we have shown that for all $\epsilon > 0$ it holds that
\begin{equation*}
\begin{split}
\norm{d_{ki}^k \delta^{ij} \A_j^U f}_{\Ltwo(G)} & \leq \sqrt{C} \norm{\gradient_\G f}_{\Ltwo(G)} \\
& \leq \epsilon \sqrt{C} \norm{\laplace_\G f}_{\Ltwo(G)} \\
& + \frac{\sqrt{C}}{4 \epsilon} \norm{f}_{\Ltwo(G)}.
\end{split}
\end{equation*}
We can now set $\epsilon = a / \sqrt{C}$ and choose $b = C / 4a$ to find
\begin{equation*}
\norm{d_{ki}^k g^{ij} \A_j^U f}_{\Ltwo(G)} \leq a \norm{\laplace_\G f}_{\Ltwo(G)} + b \norm{f}_{\Ltwo(G)},
\end{equation*}
as required. 
\end{proof}
\begin{remark}
The condition $\norm{d_{ki}^k(p)}_{\mathbb{L}_\infty(G)} < \infty$ is sufficient to ensure that there is a $C > 0$ such that $g^{ij} d_{ki}^k(p) d_{lj}^l(p) \leq C$ for all $p \in G$, which is highly reasonable in practice. 
\end{remark}
Sectoriality of Lie-Cartan Laplacians, even on Lie groups that are not unimodular, follows as a special case of Thm.~\ref{thm:data_driven_lie_cartan_laplacian_sectorial}, using that Lie groups with invariant metrics are geodesically complete \cite{Milnor1976CurvaturesGroups}.
\begin{corollary}[Lie-Cartan Laplacian Sectoriality]\label{cor:lie_cartan_laplacian_sectorial}
Let $G$ be a connected Lie group, let $\G$ be an invariant metric thereon, and let $\nu \in \R$. Then the $\nu$ Lie-Cartan Laplacian $(\laplace_{\G, \nu}, H^2(G))$ is sectorial.
\end{corollary}
These theoretical results are corroborated by our practical experience: the evolutions generated by these (data-driven) Lie-Cartan Laplacians are indeed well-posed and smoothing as one would expect from a diffusion.

\subsection{PDE Scheme}\label{sec:pde_scheme}
We will now discuss how to extend RDS filtering to $\M$, using either the invariant frame or the gauge frame. For brevity, we only write out the scheme for the invariant case; the gauge frame case can be found by replacing the frame vector fields $\{\A_i\}_i$ with $\{\A_i^U\}_i$ and the invariant metrics with data-driven invariant metrics.

The diffusion is induced by a $0$ Lie-Cartan Laplacian with respect to a diagonal metric $\G$, namely
\begin{equation}\label{eq:diffusion_scheme_exact}
\laplace_{\G_D, 0} = g_D^{11} \A_1^2 + g_D^{22} \A_2^2 + g_D^{33} \A_3^2,
\end{equation}
where $g_D^{ij}$ is the $i,j$-th component of dual metric $\G_D^{-1}$, i.e. $\G_D$ is given by $\G_D(\A_i, \A_j)= g_{ii}^D \delta_{ij} = (g_D^{ii})^{-1}\delta_{ij}$. We discretise \eqref{eq:diffusion_scheme_exact}, on a grid with spatial step size $\Delta_{xy}$ and orientational step size $\Delta_\theta$, using second order central differences, with off-grid samples computed by trilinear interpolation.

As in $\Rtwo$, morphological dilations and erosions are generated by the norm of the gradient, 
\begin{multline}\label{eq:shock_scheme_exact}
\norm{\gradient_{\G_M} U}_{\G_M}^2 \\ = g_M^{11} \abs{\A_1 U}^2 + g_M^{22} \abs{\A_2 U}^2 + g_M^{33} \abs{\A_3 U}^2,
\end{multline}
where we have again chosen a diagonal metric $\G_M$, which however may be distinct from the metric $\G_D$ used for the diffusion.
We discretise \eqref{eq:shock_scheme_exact} using a Rouy-Tourin-style upwind scheme \cite{rouy1992viscosity,bekkers2015subriemanniangeodesics}, using trilinear interpolation to compute off-grid samples.

The strength of RDS filtering in $\Rtwo$ comes from the guidance terms which allow the data to instruct whether to locally perform diffusion or shock, and dilation or erosion. As in the $\Rtwo$ case, we use the edge-preserving weight function by Charbonnier et al. \cite{charbonnier1997switch} to switch between diffusion and shock, 
\begin{equation}\label{eq:ds_switch}
g(\norm{\gradient_{\G_g} U}_\G^2) \coloneqq \sqrt{1 + \norm{\gradient_{\G_g} U}_{\G_g}^2 / \lambda^2}^{\, -1}.
\end{equation}
For the coherence-enhancing shock filter, we note that the local convexity can be determined by computing the Laplacian perpendicular to the local orientation: 
\begin{equation}\label{eq:morph_switch}
S(\laplace^\perp_{{\G_S}} U) \coloneqq S(g_S^{22} \A_2^2 U + g_S^{33} \A_3^2 U),
\end{equation}
with $S$ a sigmoidal function as in Equation~\eqref{eq:shock_R2}. Altogether, the evolution PDE is
\begin{equation}\label{eq:ds_m2}
\boxed{
\begin{aligned}
\partial_t U & = g \left(\norm{\gradient_{\G_g} U_\nu}_{\G_g}^2\right) \laplace_{\G_D} U \\
& - \left(1 - g \left(\norm{\gradient_{\G_g} U_\nu}_{\G_g}^2\right)\right) \\
    & \cdot S_\rho \left(\laplace^\perp_{\G_S} U_\sigma \right)
    \norm{\gradient_{\G_M} U }_{\G_M},
\end{aligned}
}
\end{equation}
with initial condition $U(\cdot, 0) \coloneqq \W_\psi f$ for some image $f \in \Ltwo(\Rtwo)$. Here, we have regularised the guidance terms, namely:
\begin{align*}
U_\nu  &\coloneqq G_\nu *_{\SE(2)} U, \textrm{ and } \\
S_\rho(\laplace_{\G_S}^{\perp} U_\sigma) & \coloneqq G_\rho *_{\SE(2)} S(\laplace_{\G_S}^{\perp} G_\sigma *_{\SE(2)} U),
\end{align*}
with $G_\alpha$ a spatially isotropic Gaussian kernel with scale $\alpha > 0$ and $*_{\SE(2)}$ the group convolution on $\M$ (see \cite{franken2009cedos} for details on regularisation on $\SE(2) \cong \M$). The derivatives in the guidance terms are computed using central differences and linear interpolation. Finally, we apply reflective spatial boundary conditions. 
\begin{remark}
In principle, the four metrics $\G_D$, $\G_M$, $\G_g$, and $\G_S$ can be distinct. To limit the number of parameters that need to be tuned in practice, we make some restrictions. The components of all metrics are of the form $g_{11} = \xi^2$, $g_{22} = (\xi / \zeta)^2$, and $g_{33} = 1$, with $\xi = 0.1$ fixed (see Sec.~\ref{sec:gauge_frames}). 
We fix $\zeta = 1$ for the switches $\G_g$ and $\G_S$. Hence, we only tune the anisotropy parameters $\zeta$ of the diffusion and the shock.
The metric parameters can be understood intuitively by their relation to the Reeds-Shepp car model \cite{Reeds1990OptimalBackwards,Duits2018OptimalAnalysis}, as explained in Sec.~\ref{sec:m2}.
\end{remark}
The time is discretised using forward Euler, with the following stability criterion:

\begin{theorem}[Stability of RDS filtering on \texorpdfstring{$\M$}{position-orientation space}]\label{thm:stability}
Choose timestep $\tau \leq \min\{\tau_D, \tau_S\}$, where
\begin{equation}\label{eq:timesteps}
\begin{split}
\tau_D^{-1} & \coloneqq 2 \left(\frac{g_D^{11} + g_D^{22}}{\Delta_{xy}^2} + \frac{g_D^{33}}{\Delta_\theta^2}\right), \textrm{ and } \\
\tau_S^{-1} & \coloneqq \sqrt{\frac{g_M^{11} + g_M^{22}}{\Delta_{xy}^2} + \frac{g_S^{33}}{\Delta_\theta^2}}.
\end{split}
\end{equation}
Then, this scheme satisfies a maximum-minimum principle:
\begin{equation}\label{eq:stability}
\min_{w, h, o} U_{w, h, o}^0 \leq U_{i, j, k}^t \leq \max_{w, h, o} U_{w, h, o}^0 \textrm{ for all } i, j, k, t,
\end{equation}
with $U^t$ the solution at time step $t \in \N_{\geq 0}$.
\end{theorem}
\begin{proof}
The proof, which is analogous to that of the stability result for RDS filtering on $\Rtwo$ \cite[Thm.~1]{schaefer2024regularisedds}, can be found in App.~\ref{app:proof_stability}.
\end{proof}

\definecolor{oiBlue}{RGB}{0,114,178}
\definecolor{oiRed}{RGB}{213,94,0}
\definecolor{oiGreen}{RGB}{0,158,115}
\definecolor{oiYellow}{RGB}{240,228,66}
\definecolor{oiPurple}{RGB}{204,121,167}
\definecolor{oiBlack}{RGB}{0,0,0}
\definecolor{oiGrey}{RGB}{160,160,160}

\pgfplotscreateplotcyclelist{mycycle}{%
  {oiGreen,solid,mark=none},
  {oiGreen,dashed,mark=none},
  {oiBlue,solid,mark=none},
  {oiRed,solid,mark=none},
  {oiRed,dashed,mark=none},
  {oiYellow,solid,mark=none},
  {oiPurple,solid,mark=none},
}

\pgfplotsset{
  myaxis/.style={
    width=0.53\linewidth,
    height=7cm,
    xlabel={$t / t_\mathrm{opt}$},
    xmin=0, xmax=2,
    xtick={0,1,2},
    enlargelimits=false,
    no markers,
    line width=1.2pt,
  },
}

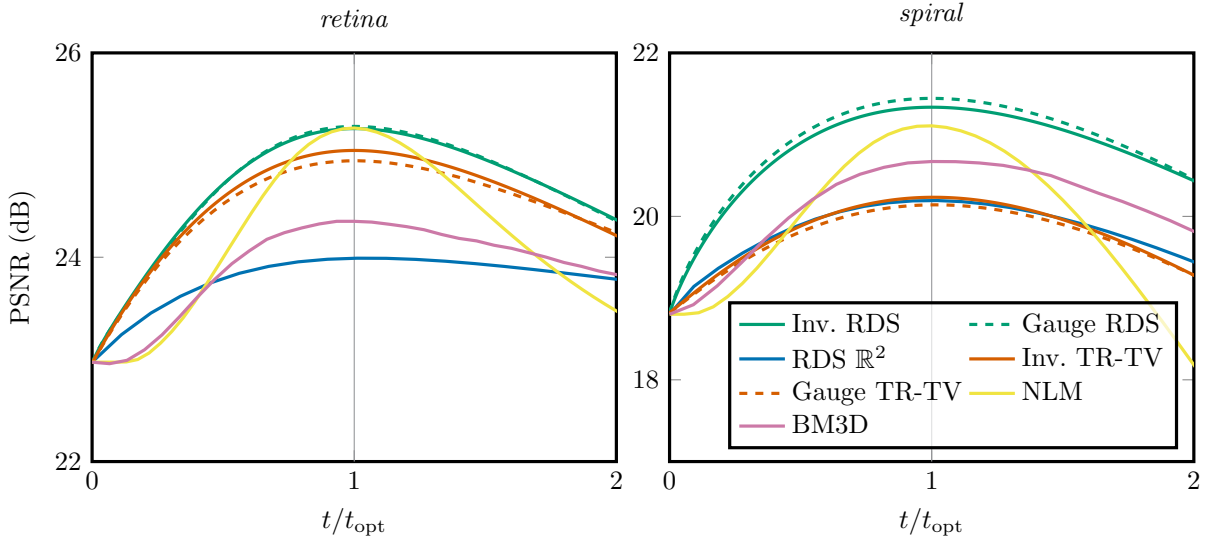
\begin{figure*}
\centering
\begin{tikzpicture}[baseline=(current bounding box.north west)]
\begin{axis}[
  name=retina,
  myaxis,
  ylabel={PSNR (dB)},
  title={\emph{retina}},
  ymin=22, ymax=26,
  ytick={22,24,26},
  cycle list name=mycycle,
]
\addplot[Gray,thin,forget plot] coordinates {(1,22) (1,26)};

\addplot table[x index=0, y index=1] {Figures/experiments/retina_DS_LI.dat};
\addplot table[x index=0, y index=1] {Figures/experiments/retina_DS_gauge.dat};
\addplot table[x index=0, y index=1] {Figures/experiments/retina_DS_R2.dat};
\addplot table[x index=0, y index=1] {Figures/experiments/retina_TV_LI.dat};
\addplot table[x index=0, y index=1] {Figures/experiments/retina_TV_gauge.dat};
\addplot table[x index=0, y index=1] {Figures/experiments/retina_nlm.dat};
\addplot table[x index=0, y index=1] {Figures/experiments/retina_bm3d.dat};
\end{axis}

\begin{axis}[
  name=spiral,
  myaxis,
  at={(retina.east)}, anchor=west, xshift=0.7cm,
  title={\emph{spiral}},
  ymin=17, ymax=22,
  ytick={18,20,22},
  legend pos=south east,
  legend style={
    legend columns=2,
    fill=white, fill opacity=0.7,
    draw opacity=1, text opacity=1,
    cells={anchor=west},
  },
  cycle list name=mycycle,
]
\addplot[Gray,thin,forget plot] coordinates {(1,17) (1,22)};

\addplot table[x index=0, y index=1] {Figures/experiments/spiral_DS_LI.dat};
\addplot table[x index=0, y index=1] {Figures/experiments/spiral_DS_gauge.dat};
\addplot table[x index=0, y index=1] {Figures/experiments/spiral_DS_R2.dat};
\addplot table[x index=0, y index=1] {Figures/experiments/spiral_TV_LI.dat};
\addplot table[x index=0, y index=1] {Figures/experiments/spiral_TV_gauge.dat};
\addplot table[x index=0, y index=1] {Figures/experiments/spiral_nlm.dat};
\addplot table[x index=0, y index=1] {Figures/experiments/spiral_bm3d.dat};

\legend{
  Inv.\ RDS,
  Gauge RDS,
  RDS $\mathbb{R}^2$,
  Inv.\ TR-TV,
  Gauge TR-TV,
  NLM,
  BM3D
}
\end{axis}
\end{tikzpicture}
\caption{PSNRs of denoising methods over time relative to optimal stopping time.}
\label{fig:plots}
\end{figure*}

\subsection{Experimental Results}\label{sec:experiments}
RDS filtering is a promising technique for inpainting, as already shown in \cite{SW23,schaefer2024regularisedds}, and for image enhancement, due to its combination of diffusion for denoising and shock for sharpening.
Here we show a selection of image enhancement and inpainting experiments performed using RDS filtering on $\M$. Our Python implementations, which use Taichi \cite{taichi} for GPU acceleration, and experiments can be found at \url{https://github.com/finnsherry/M2RDSFiltering}.
\subsubsection{Image Enhancement}\label{sec:denoising}
We compare the enhancement of degraded images using gauge and invariant RDS filtering on $\M$ to our own implementations of RDS filtering on $\Rtwo$ \cite{schaefer2024regularisedds} and gauge and invariant TR-TV \cite{smets2021tvflow}; the Python implementation of BM3D \cite{bm3dold} in the package \texttt{bm3d}; and the Python implementation of NLM \cite{nlm} in the package \texttt{scikit-image} \cite{scikit-image}. We consider a medical image \cite{Maastrichtstudy} -- a patch of the retina -- in Figure~\ref{fig:retina}, and a cartoon-like image of overlapping spirals in Figure~\ref{fig:spiral}.  
In both cases the images have been degraded with additive, correlated noise $K_\rho * n_\sigma$, with $K_\rho$ a Gaussian of standard deviation $\rho$ and $n_\sigma$ white noise with intensity $\sigma$; in Figure~\ref{fig:retina} we have $(\sigma, \rho) = (127.5, 2)$ and in Figure~\ref{fig:spiral} we have $(\sigma, \rho) = (255, 2)$. These experiments can be reproduced with the notebook \href{https://github.com/finnsherry/M2RDSFiltering/blob/main/experiments/image_enhancement.ipynb}{image\_enhancement.ipynb}.

\begin{figure*}
\centering
\begin{tabular}{ccc}
\includegraphics[width=0.3\linewidth,frame]{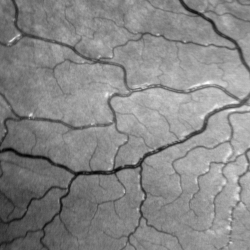} &
\includegraphics[width=0.3\linewidth,frame]{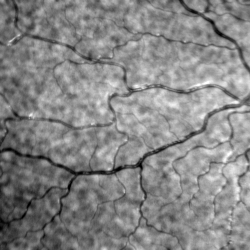}&
\includegraphics[width=0.3\linewidth,frame]{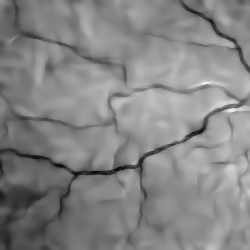} \\
(a) Clean image & (b) Degraded image & (c) Invariant $\M$ RDS \\
\includegraphics[width=0.3\linewidth,frame]{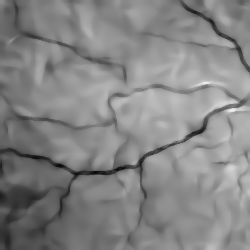} &
\includegraphics[width=0.3\linewidth,frame]{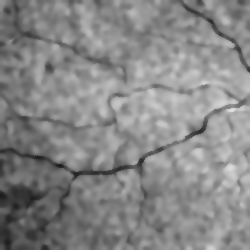} &
\includegraphics[width=0.3\linewidth,frame]{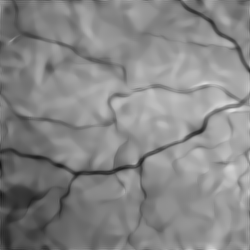} \\
(d) Gauge $\M$ RDS & (e) $\Rtwo$ RDS & (f) Invariant $\M$ TR-TV \\
\includegraphics[width=0.3\linewidth,frame]{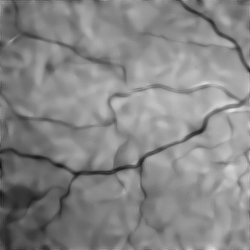} &
\includegraphics[width=0.3\linewidth,frame]{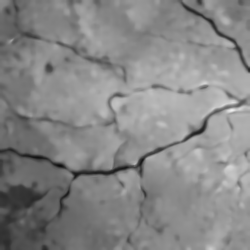} &
\includegraphics[width=0.3\linewidth,frame]{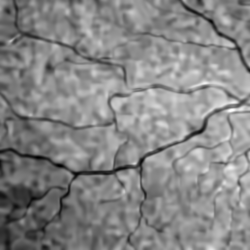} \\
(g) Gauge $\M$ TR-TV & (h) NLM & (i) BM3D
\end{tabular}
\caption{Denoising of image \emph{retina} degraded with correlated noise.}
\label{fig:retina}
\end{figure*}

We quantify the quality of the denoising using the Peak Signal-to-Noise Ratio (PSNR), given by
\begin{multline}\label{eq:psnr}
\PSNR(f, g) \coloneqq \\ 10 \cdot \log_{10}\left(\frac{255}{\int_{\Rtwo} \abs{f(\vec{x}) - g(\vec{x})}^2 \diff \vec{x}}\right) \mathrm{dB},
\end{multline}
with $f$ the denoised image and $g$ the ground truth.
Figure~\ref{fig:plots} shows the PSNR as a function of stopping time normalized by the optimal stopping time for the PDE-based methods (RDS, TR-TV), and as a function of noise power normalized to the optimal noise power for the other methods (BM3D, NLM).
RDS filtering on $\M$ outperforms the other methods w.r.t. maximal PSNR, and is less sensitive to the stopping time than BM3D and NLM. 
In Figure~\ref{fig:retina} and \ref{fig:spiral} we qualitatively compare the results of various methods at their highest PSNRs.
We see in Figure~\ref{fig:retina} that the PDE-based methods preserve the small vessels better than NLM.
In Figure~\ref{fig:spiral}, we see that NLM has enhanced spurious blobs on the background which have been removed by $\M$ RDS.
Gauge frame RDS performs better than invariant RDS, whereas for TR-TV the invariant version outperforms the gauge frame version, although in both cases the differences are slight.

\begin{figure*}
\centering
\begin{tabular}{ccc}
\includegraphics[width=0.3\linewidth,frame]{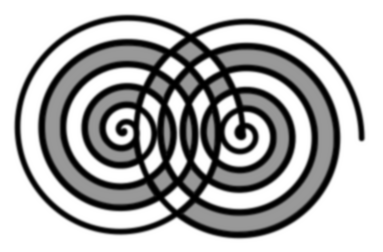} &
\includegraphics[width=0.3\linewidth,frame]{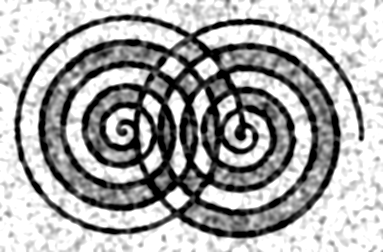}&
\includegraphics[width=0.3\linewidth,frame]{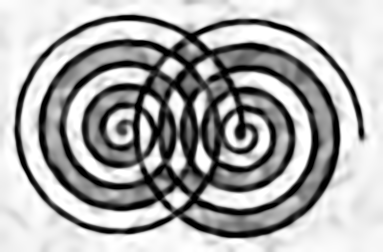} \\
(a) Clean image & (b) Degraded image & (c) LI $\M$ RDS \\
\includegraphics[width=0.3\linewidth,frame]{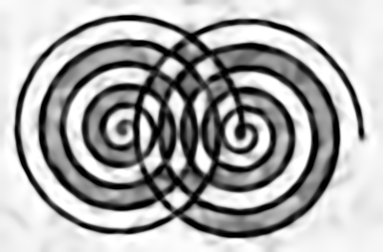} &
\includegraphics[width=0.3\linewidth,frame]{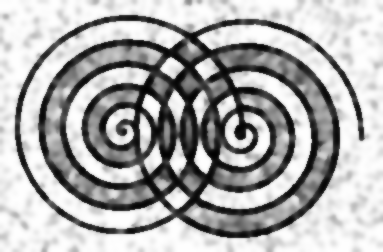} &
\includegraphics[width=0.3\linewidth,frame]{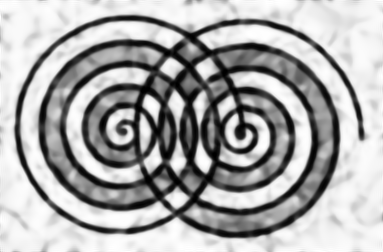} \\
(d) Gauge $\M$ RDS & (e) $\Rtwo$ RDS & (f) LI $\M$ TR-TV \\
\includegraphics[width=0.3\linewidth,frame]{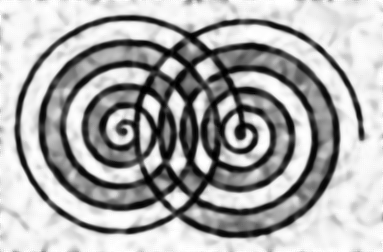} &
\includegraphics[width=0.3\linewidth,frame]{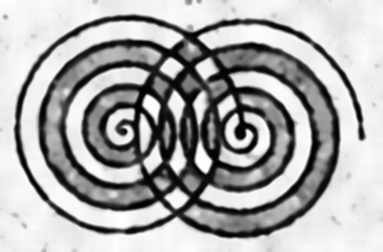} &
\includegraphics[width=0.3\linewidth,frame]{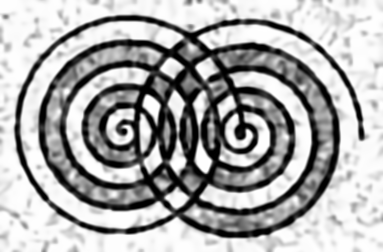} \\
(g) Gauge $\M$ TR-TV & (h) NLM & (i) BM3D
\end{tabular}
\caption{Denoising of image \emph{spiral} degraded with correlated noise.}
\label{fig:spiral}
\end{figure*}

\subsubsection{Inpainting}
RDS filtering on $\Rtwo$ was originally developed with the goal of inpainting, the task of filling in missing gaps in an image \cite{SW23,schaefer2024regularisedds}. 
While it can create good inpainting results, RDS inpainting on $\Rtwo$ -- like many other methods on $\Rtwo$ -- cannot reconstruct crossings. 
However, by lifting the image to $\M$ with the orientation score transform \eqref{eq:ost}, crossing structures are disentangled (recall Fig.~\ref{fig:multiorientation_processing}), allowing RDS inpainting on $\M$ to reconstruct crossings. 
Figure~\ref{fig:inpainting} demonstrates this: (invariant) RDS inpainting Figure~\ref{fig:inpainting}c on $\M$ creates crossing lines, while on $\Rtwo$ it connects different lines without crossings Figure~\ref{fig:inpainting}b. These experiments can be reproduced with the notebook \href{https://github.com/finnsherry/M2RDSFiltering/blob/main/experiments/inpainting.ipynb}{inpainting.ipynb}.

\begin{figure*}[tb]
\centering
\begin{tabular}{ccc}
\includegraphics[width=0.3\linewidth,frame]{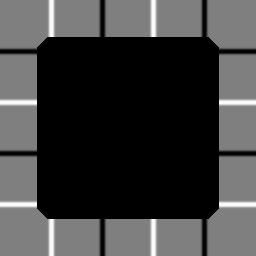} &
\includegraphics[width=0.3\linewidth,frame]{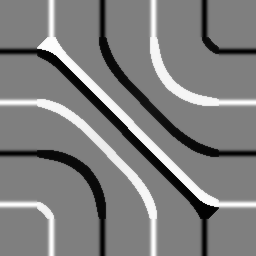}&
\includegraphics[width=0.3\linewidth,frame]{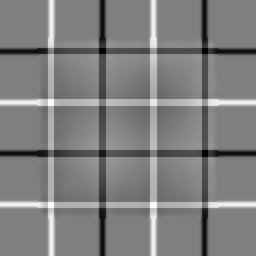}  \\
(a) Masked input & (b) $\Rtwo$ RDS inpainting &  (c) LI $\M$ RDS inpainting \\
\end{tabular}
\caption{RDS inpainting of crossings on $\Rtwo$ and $\M$. The inpainting region is marked by a black square in (a).}
\label{fig:inpainting}
\end{figure*}

\section{Training Diffusion-Shock Filtering in PDE-Based Networks}\label{sec:trained_diffusion_shock_filtering}
In Sec.~\ref{sec:experiments}, RDS filtering showed promising results as an image processing technique. One limitation of RDS filtering, which it shares with other PDE-based methods, is that it requires manual parameter tuning. To overcome this, we propose to integrate RDS into PDE-based networks, such that the parameters can be learned from data. 

\subsection{Scheme}
As described in Sec.~\ref{sec:pde_based_networks}, the primary PDEs currently used in PDE-based networks generate generalised scale spaces, and hence can be solved efficiently using generalised convolutions. Unfortunately, the RDS PDE does not generate such a scale space. We therefore developed a computationally tractable scheme based on operator splitting and gating, which we explain here for a single layer and channel. 
Gating, seen e.g. in the gated recurrent unit used in recurrent neural networks \cite{Cho2014PropertiesApproaches}, essentially means that the processing that is performed depends on the input. In this sense, we can actually interpret the terms $g$ and $S$ in the RDS PDE as location dependent gating mechanisms. The $g$ term ensures diffusion occurs when there is no structure present, while shock occurs when there is structure. Similarly, $S$ will cause light areas to be dilated while dark areas get eroded. We therefore compute the diffusions, dilations, and erosions -- which can be done using existing modules in the Python package LieTorch \cite{Smets2022PDEGCNNs} -- and combine them using these new guidance terms.

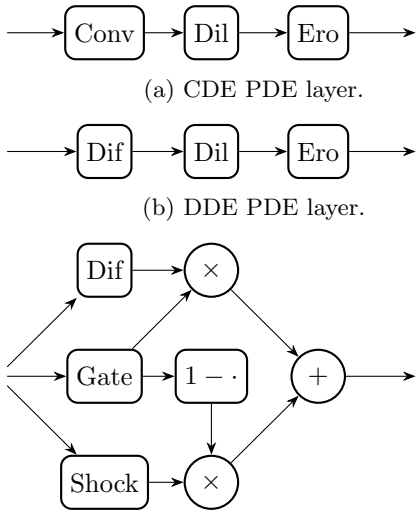
\begin{figure}
\centering
\begin{subfigure}[t]{0.9\linewidth}
\begin{tikzpicture}[
    node distance=4em,
    circ/.style={circle,draw,thick,minimum size=2em,align=center},
    sqr/.style={rectangle,draw,thick,rounded corners,minimum size=2em,align=center},
    plain/.style={draw=none,align=center},
    ->, >=Stealth
]

\node[plain] (U) {};
\node[sqr, right of=U] (C) {Conv};
\node[sqr, right of=C] (D) {Dil};
\node[sqr, right of=D] (E) {Ero};
\node[plain, right of=E] (out) {};

\draw (U) -- (C);
\draw (C) -- (D);
\draw (D) -- (E);
\draw (E) -- (out);
\end{tikzpicture}
\caption{CDE PDE layer.}\label{fig:cde_pde_layer}
\end{subfigure}
\begin{subfigure}[t]{0.9\linewidth}
\begin{tikzpicture}[
    node distance=4em,
    circ/.style={circle,draw,thick,minimum size=2em,align=center},
    sqr/.style={rectangle,draw,thick,rounded corners,minimum size=2em,align=center},
    plain/.style={draw=none,align=center},
    ->, >=Stealth
]

\node[plain] (U) {};
\node[sqr, right of=U] (Dif) {Dif};
\node[sqr, right of=Dif] (D) {Dil};
\node[sqr, right of=D] (E) {Ero};
\node[plain, right of=E] (out) {};

\draw (U) -- (Dif);
\draw (Dif) -- (D);
\draw (D) -- (E);
\draw (E) -- (out);
\end{tikzpicture}
\caption{DDE PDE layer.}\label{fig:dde_pde_layer}
\end{subfigure}
\begin{subfigure}[t]{0.9\linewidth}
\begin{tikzpicture}[
    node distance=4em,
    circ/.style={circle,draw,thick,minimum size=2em,align=center},
    sqr/.style={rectangle,draw,thick,rounded corners,minimum size=2em,align=center},
    plain/.style={draw=none,align=center},
    ->, >=Stealth
]

\node[plain] (U) {};
\node[sqr, right of=U] (g) {Gate};
\node[sqr, below of=g] (M) {Shock};
\node[sqr, above of=g] (D) {Dif};

\node[circ, right of=D] (tD) {$\times$};
\node[sqr, right of=g] (flip) {$1 - \cdot$};
\node[circ, right of=M] (tM) {$\times$};

\node[circ, right of=flip] (pD) {$+$};

\node[plain, right of=pD] (out) {};

\draw (U) -- (D);
\draw (U) -- (M);
\draw (U) -- (g);

\draw (D) -- (tD);
\draw (g) -- (tD);

\draw (M) -- (tM);
\draw (g) -- (flip);
\draw (flip) -- (tM);

\draw (tD) -- (pD);
\draw (tM) -- (pD);
\draw (pD) -- (out);
\end{tikzpicture}
\caption{RDS PDE layer, implementing Eq.~\eqref{eq:ds_trained} to approximately solve Eq.~\eqref{eq:ds_m2}.}\label{fig:rds_pde_layer}
\end{subfigure}
\caption{Comparison of PDE-(G-)CNN layers. The labeled nodes perform the following operations: Conv $\leftrightarrow$ convection; Dil $\leftrightarrow$ dilation; Ero $\leftrightarrow$ erosion; Dif $\leftrightarrow$ diffusion; Gate $\leftrightarrow$ gating mechanism $g$.}
\label{fig:pde_layer_comparison}
\end{figure}

The overall structure of the RDS PDE layer can be seen in Fig.~\ref{fig:rds_pde_layer}; here we describe the layer mathematically. Note that in PDE-G-CNNs, feature maps are independently evolved by PDEs before being affinely mixed (recall Fig.~\ref{fig:pde_g_cnn}); as such we only consider the action of PDE layers on individual feature maps.

We first introduce some notation for the solution operators. $\Phi_\cdot^\mathrm{dif}(U_0)$ denotes the solution of the diffusion equation $\partial_t U = \laplace_{\G_\cdot} U$,
while $\Phi_\cdot^\mathrm{dil}(U_0)$ is the solution of the dilation equation $\partial_t U = \norm{\gradient_{\G_\cdot} U}_{\G_\cdot}^{2 \alpha}$,\footnote{$\alpha \in (1/2, 1]$ controls the smoothness of the kernel; see \cite[Rem.~14]{Smets2022PDEGCNNs} for details. The kernels must be smooth to enable network parameter optimisation.}
with initial condition $U_0$ at some fixed time $T > 0$. Note that then $-\Phi_\cdot^\mathrm{dil}(-U_0)$ is the solution of the erosion equation $\partial_t U = -\norm{\gradient_{\G_\cdot} U}_{\G_\cdot}^{2 \alpha}$, with initial condition $U_0$ at time $T$ \cite{Smets2022PDEGCNNs}. Next, we define the guidance terms:
\begin{equation*}
\begin{split}
\Phi_\lambda^g(U) & \coloneqq g(\Phi_g^\mathrm{dif}(\norm{\gradient_{\G_g} U}_{\G_g}); \lambda), \textrm{ and } \\
\Phi_\epsilon^S(U) & \coloneqq S(\laplace_{\G_S} \Phi_S^\mathrm{dif}(U); \epsilon).
\end{split}
\end{equation*}
Note that we regularise by diffusing, instead of by convolving with a spatially isotropic Gaussian kernel as in Sec.~\ref{sec:m2_ds}, allowing for spatially anisotropic regularisation in $\M$. We then combine these components as follows:
\begin{equation}\label{eq:ds_trained}
\begin{split}
& \Phi(U; \G_D, \G_M, \G_g, \lambda, \G_S, \epsilon) \coloneqq \Phi_\lambda^g(U) \cdot \Phi_D^\mathrm{dif}(U) \\
& \quad + (1 - \Phi_\lambda^g(U)) \cdot \abs{\Phi_\epsilon^S(U)} \\
& \qquad \cdot \big(\Phi_M^\mathrm{dil}(U) \cdot \indicator\{\Phi_\epsilon^S(U) < 0\} \\
& \qquad - \Phi_M^\mathrm{dil}(-U) \cdot \indicator\{\Phi_\epsilon^S(U) > 0\} \big) \\
& \quad + (1 - \Phi_\lambda^g(U)) (1 - \abs{\Phi_\epsilon^S(U)}) \cdot U.
\end{split}
\end{equation}
Here, $U$ is the input feature map and $\G_D$, $\G_M$, $\G_g$, $\lambda$, $\G_S$, and $\epsilon$ are the trainable parameters of the RDS PDE layer. 
We can reason intuitively about this definition in some limiting cases:
\begin{itemize}
\item In regions where $\Phi_\lambda^g(U) \approx 1$, only diffusion occurs;
\item In regions where $\Phi_\lambda^g(U) \approx 0$, only shock occurs: dilation where $\Phi_\epsilon^S(U) < 0$ and erosion where $\Phi_\epsilon^S(U) > 0$;
\item In regions where $\Phi_\lambda^g(U) \approx 0$ and $\Phi_\epsilon^S(U) \approx 0$, neither diffusion nor shock occurs, so $U$ should remain unchanged. 
\end{itemize}
While the approximations of the diffusion, dilation, and erosion have already been studied in \cite{Smets2022PDEGCNNs}, we have not quantified the accuracy of our scheme: it would be interesting to investigate how the approximation error depends on the time step $T$.

In addition to the metric parameters, for which we use the same initialisation as for previous PDE-(G-)CNNs \cite{Smets2022PDEGCNNs,Bellaard2025PDECNNs}, the guidance terms give us two new trainable parameters: the contrast parameter $\lambda$ and the regularisation parameter $\epsilon$. For the PDE-CNN version, we initialise $\lambda \sim \Uniform[-1, 1]$ and $\epsilon \sim \Uniform[-1, 1]$; for PDE-G-CNNs, we initialise $\lambda \sim \Uniform[-10, 10]$ and $\epsilon \sim \Uniform[-10, 10]$.

\subsection{Experiments}\label{sec:experiments_trained}
\begin{figure*}
\centering
\begin{tabular}{cccc}
\includegraphics[width=0.22\linewidth]{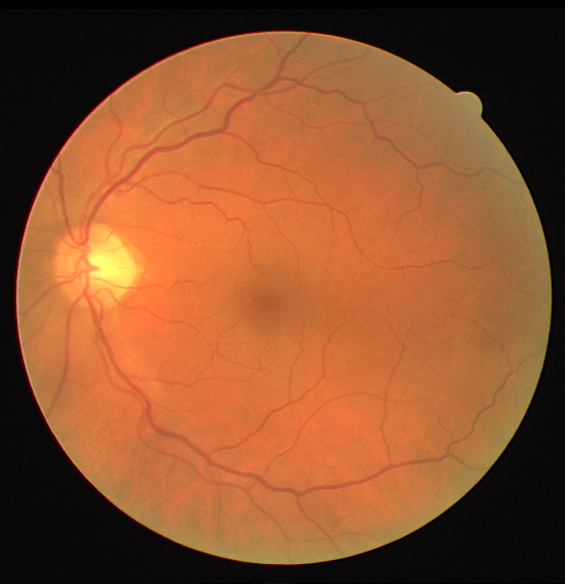} &
\includegraphics[width=0.22\linewidth]{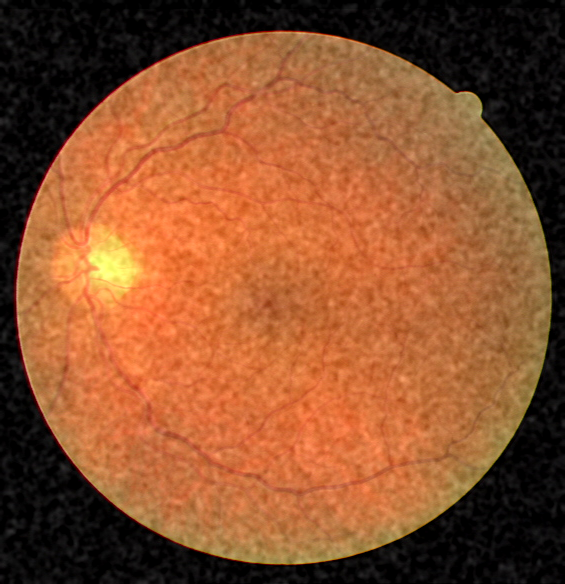} &
\includegraphics[width=0.22\linewidth]{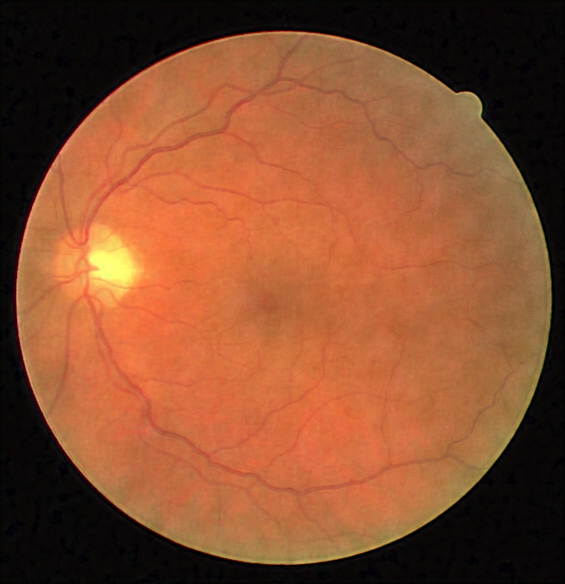} &
\includegraphics[width=0.22\linewidth]{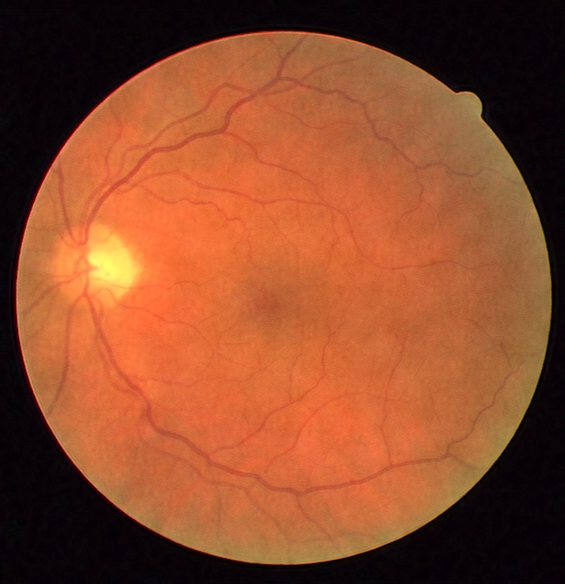} \\
(a) Clean image & (b) Degraded image & (c) RDS $\M$ & (d) RDS $\Rtwo$ \\
\includegraphics[width=0.22\linewidth]{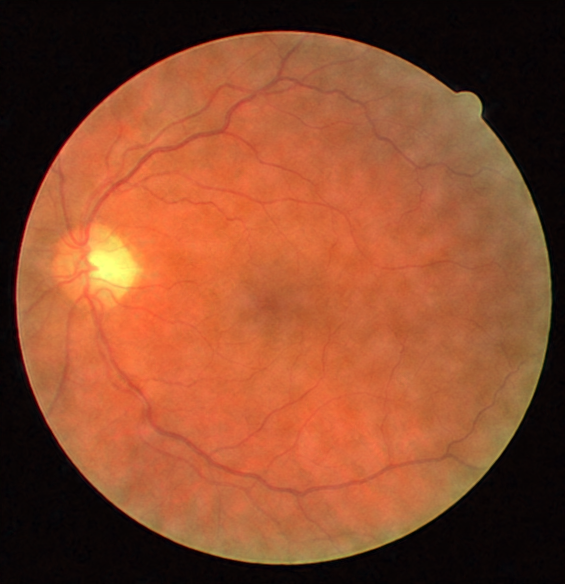} &
\includegraphics[width=0.22\linewidth]{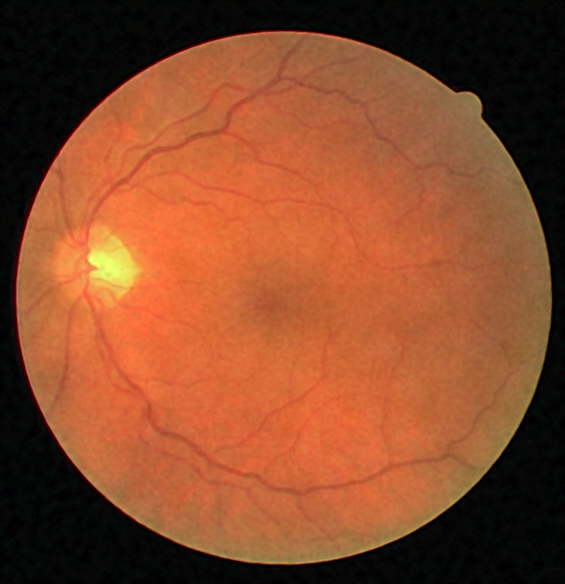} &
\includegraphics[width=0.22\linewidth]{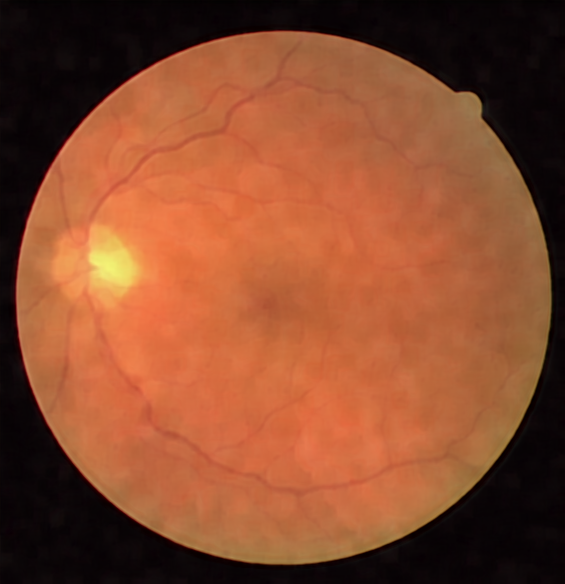} &
\includegraphics[width=0.22\linewidth]{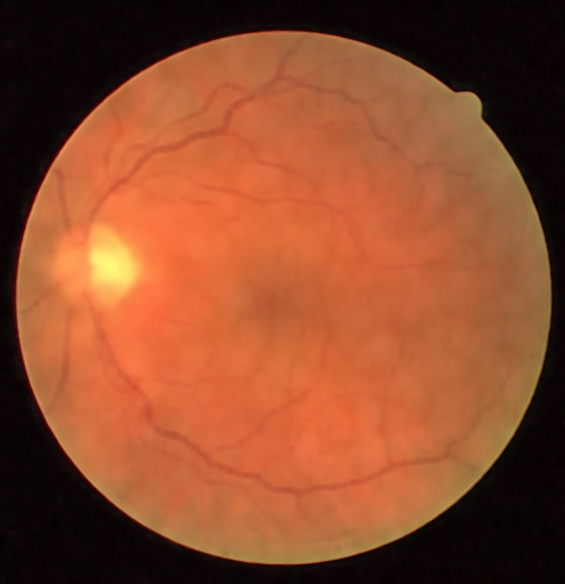} \\
(e) CDE $\M$ & (f) DDE $\M$ & (g) CDE $\Rtwo$ & (h) DDE $\Rtwo$ \\
\end{tabular}
\caption{Sample from the DRIVE dataset \cite{Staal2004RidgeRetina}. (a) Clean image; (b) input image to be denoised; (c-h) output from trained PDE-(G-)CNNs.}
\label{fig:drive}
\end{figure*}
\begin{figure*}
\centering
\begin{tikzpicture}[baseline=(current bounding box.west)]
\begin{axis}[
    name=psnr,
    boxplot/draw direction=x,
    height=5cm,
    width=0.4\linewidth,
    title={PSNR (dB)},
    xmin=15, xmax=30,
    xtick={15, 20, 25, 30},
    ytick={1, 2, 3, 4, 5, 6},
    yticklabels={DDE $\Rtwo$, CDE $\Rtwo$, RDS $\Rtwo$, DDE $\M$, CDE $\M$, RDS $\M$}
]

\addplot+[
    boxplot prepared={
        lower whisker=18.82238197,
        lower quartile=20.44015645980835,
        median=22.52833843231201,
        upper quartile=22.938190937042236,
        upper whisker=23.91912841796875
    },
    black,
    solid,
    thick
] coordinates {};

\addplot+[
    boxplot prepared={
        lower whisker=18.807395935058594,
        lower quartile=19.710847854614258,
        median=21.568166732788086,
        upper quartile=22.70248794555664,
        upper whisker=23.34989356994629
    },
    black,
    solid,
    thick
] coordinates {};

\addplot+[
    boxplot prepared={
        lower whisker=18.037418365478516,
        lower quartile=19.46623468399048,
        median=21.373231887817383,
        upper quartile=24.03748083114624,
        upper whisker=24.704116821289062
    },
    OliveGreen,
    solid,
    thick
] coordinates {};

\addplot+[
    boxplot prepared={
        lower whisker=23.98473549,
        lower quartile=26.02223014831543,
        median=27.022764205932617,
        upper quartile=27.19014072418213,
        upper whisker=27.557680130004883
    },
    black,
    solid,
    thick
] coordinates {};

\addplot+[
    boxplot prepared={
        lower whisker=24.041229248046875,
        lower quartile=26.248781204223633,
        median=27.124736785888672,
        upper quartile=27.45819616317749,
        upper whisker=27.987688064575195
    },
    black,
    solid,
    thick
] coordinates {};

\addplot+[
    boxplot prepared={
        lower whisker=25.83173370361328,
        lower quartile=26.60090208053589,
        median=27.429889678955078,
        upper quartile=29.314905643463135,
        upper whisker=29.84362793
    },
    OliveGreen,
    solid,
    thick
] coordinates {};

\end{axis}

\begin{axis}[
    name=ssim,
    at={(psnr.east)},
    anchor=west,
    xshift=1cm,
    clip=false,
    boxplot/draw direction=x,
    height=5cm,
    width=0.4\linewidth,
    title={SSIM},
    xmin=0.65, xmax=0.95,
    xtick={0.65, 0.75, 0.85, 0.95},
    ytick={1, 2, 3, 4, 5, 6},
    yticklabels={}
]

\addplot+[
    boxplot prepared={
        lower whisker=0.6860532164573669,
        lower quartile=0.7392765134572983,
        median=0.7661594152450562,
        upper quartile=0.7912729680538177,
        upper whisker=0.8480874300003052
    },
    black,
    solid,
    thick
] coordinates {};

\addplot+[
    boxplot prepared={
        lower whisker=0.6854907274246216,
        lower quartile=0.7703386098146439,
        median=0.7843812108039856,
        upper quartile=0.8180289715528488,
        upper whisker=0.8561238050460815
    },
    black,
    solid,
    thick
] coordinates {};

\addplot+[
    boxplot prepared={
        lower whisker=0.7038760185241699,
        lower quartile=0.755884439,
        median=0.7744369804859161,
        upper quartile=0.8123606443405151,
        upper whisker=0.8605347871780396
    },
    OliveGreen,
    solid,
    thick
] coordinates {};

\addplot+[
    boxplot prepared={
        lower whisker=0.738758922,
        lower quartile=0.8339939564466476,
        median=0.8454535007476807,
        upper quartile=0.8513462841510773,
        upper whisker=0.8679925799369812
    },
    black,
    solid,
    thick
] coordinates {};

\addplot+[
    boxplot prepared={
        lower whisker=0.7915983200073242,
        lower quartile=0.8411989361047745,
        median=0.8599987030029297,
        upper quartile=0.870142475,
        upper whisker=0.8927407264709473
    },
    black,
    solid,
    thick
] coordinates {};

\addplot+[
    boxplot prepared={
        lower whisker=0.8161311149597168,
        lower quartile=0.8430549502372742,
        median=0.8835571706295013,
        upper quartile=0.8959725350141525,
        upper whisker=0.9029864072799683
    },
    OliveGreen,
    solid,
    thick
] coordinates {};

\node[black] at (axis cs:1.04,7.91){Parameters};
\node[OliveGreen] at (axis cs:1.04,6){4160};
\node[black] at (axis cs:1.04,5){4160};
\node[black] at (axis cs:1.04,4){4160};
\node[OliveGreen] at (axis cs:1.04,3){4262};
\node[black] at (axis cs:1.04,2){4126};
\node[black] at (axis cs:1.04,1){4058};

\end{axis}

\end{tikzpicture}
\caption{Comparison of PDE-(G-)CNNs tested on DRIVE. The new architectures based on RDS filtering are highlighted in \textcolor{OliveGreen}{green}.}
\label{fig:trained_denoising}
\end{figure*}
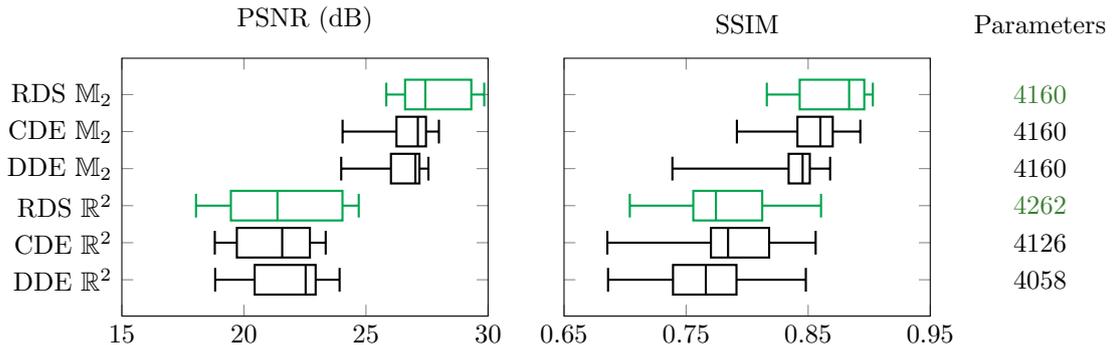
Our implementations of RDS PDE-(G-)CNNs have been added to the open source Python package LieTorch \cite{Smets2022PDEGCNNs}, available at \url{https://gitlab.com/bsmetsjr/lietorch}. Here, we again test the efficacy of the methods on denoising and inpainting tasks. Previous work \cite{Smets2022PDEGCNNs,Bellaard2023Analysis,Bellaard2023GeometricPDEGCNNs,Pai2023FunctionalNetworks,Bellaard2025PDECNNs} has shown that PDE-(G-)CNNs can compete with (G-)CNNs in terms of performance, while significantly reducing the network complexity and improving the data efficiency. Hence, we focus on comparing to existing PDE-(G-)CNNs at fixed network complexity. We use two PDEs as a baseline: 
\begin{enumerate}
\item Convection-Dilation-Erosion (CDE): This PDE (see Fig.~\ref{fig:cde_pde_layer}) was frequently used in the past.
\item Diffusion-Dilation-Erosion (DDE): This PDE (see Fig.~\ref{fig:dde_pde_layer}) is most similar to RDS, and hence allows us to test the gating mechanism.
\end{enumerate}
For the PDE-CNNs, we constrain the metrics to be isotropic, such that the corresponding PDE layers are approximately equivariant to roto-translations. Notably, as only the trivial convection on $\Rtwo$ is equivariant, the CDE PDE-CNN therefore still won't be equivariant.

In all experiments we use the AdamW optimiser \cite{Loshchilov2017DecoupledRegularization} with an exponential learning rate decay schedule. We set the initial learning rate to 0.01, the decay rate to 0.95, and the weight decay to 0.005; for the rest of the parameters we use the PyTorch default values. In all PDE-G-CNNs we lift to 8 orientations. The parametrised kernels have shape $7 \times 7$ in the PDE-CNNs and $7 \times 7 \times 7$ in the PDE-G-CNNs.
Every model architecture is trained and tested 10 times with a different seed.

\begin{figure*}
\centering
\begin{tabular}{cccc}
\includegraphics[width=0.22\linewidth]{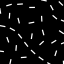} &
\includegraphics[width=0.22\linewidth]{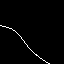} &
\includegraphics[width=0.22\linewidth]{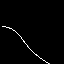} &
\includegraphics[width=0.22\linewidth]{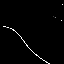} \\
(a) Input & (b) Ground truth & (c) RDS $\M$ & (d) RDS $\Rtwo$ \\
\includegraphics[width=0.22\linewidth]{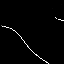} &
\includegraphics[width=0.22\linewidth]{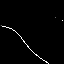} &
\includegraphics[width=0.22\linewidth]{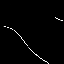} &
\includegraphics[width=0.22\linewidth]{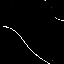}  \\
(e) CDE $\M$ & (f) DDE $\M$ & (g) CDE $\Rtwo$ & (h) DDE $\Rtwo$ \\
\end{tabular}
\caption{Sample from the Lines dataset \cite{Bellaard2023Analysis}. (a) Input image in which the line must be completed; (b) ground truth line; (c-h) output from trained PDE-(G-)CNNs.}
\label{fig:lines}
\end{figure*}
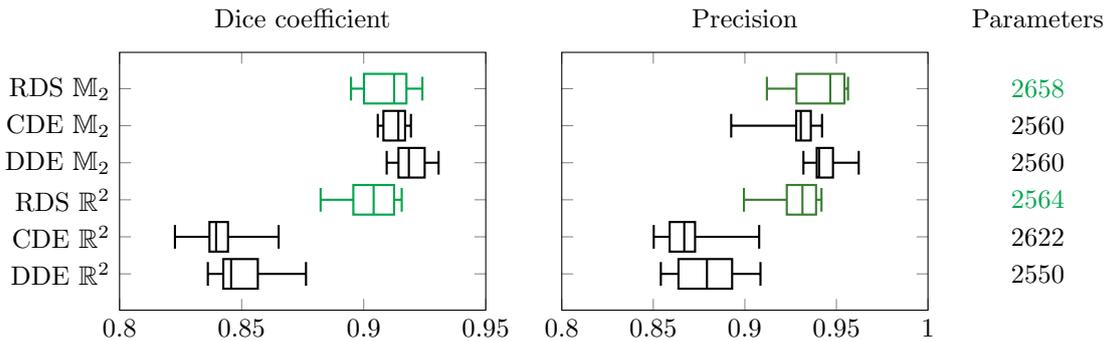
\begin{figure*}
\centering
\begin{tikzpicture}[baseline=(current bounding box.west)]
\begin{axis}[
    name=dice,
    boxplot/draw direction=x,
    height=5cm,
    width=0.4\linewidth,
    title={Dice coefficient},
    xmin=0.8, xmax=0.95,
    xtick={0.8, 0.85, 0.9, 0.95},
    ytick={1, 2, 3, 4, 5, 6},
    yticklabels={DDE $\Rtwo$, CDE $\Rtwo$, RDS $\Rtwo$, DDE $\M$, CDE $\M$, RDS $\M$}
]

\addplot+[
    boxplot prepared={
        lower whisker=0.836161196,
        lower quartile=0.8425013422966003,
        median=0.8456999957561493,
        upper quartile=0.8566375970840454,
        upper whisker=0.8763683438301086
    },
    black,
    solid,
    thick
] coordinates {};

\addplot+[
    boxplot prepared={
        lower whisker=0.8226860761642456,
        lower quartile=0.8368242233991623,
        median=0.8395455479621887,
        upper quartile=0.8444167077541351,
        upper whisker=0.8651645183563232
    },
    black,
    solid,
    thick
] coordinates {};

\addplot+[
    boxplot prepared={
        lower whisker=0.8823494911193848,
        lower quartile=0.8957525193691254,
        median=0.9040854275226593,
        upper quartile=0.9124542027711868,
        upper whisker=0.9155813455581665
    },
    OliveGreen,
    solid,
    thick
] coordinates {};

\addplot+[
    boxplot prepared={
        lower whisker=0.9094266891479492,
        lower quartile=0.9142415076494217,
        median=0.9184966385364532,
        upper quartile=0.9249799996614456,
        upper whisker=0.930678666
    },
    black,
    solid,
    thick
] coordinates {};

\addplot+[
    boxplot prepared={
        lower whisker=0.9057742357254028,
        lower quartile=0.9080718606710434,
        median=0.9141895771026611,
        upper quartile=0.9169504940509796,
        upper whisker=0.9193519949913025
    },
    black,
    solid,
    thick
] coordinates {};

\addplot+[
    boxplot prepared={
        lower whisker=0.8947964906692505,
        lower quartile=0.9001179784536362,
        median=0.9124675989151001,
        upper quartile=0.9175031781196594,
        upper whisker=0.9240387678146362
    },
    OliveGreen,
    solid,
    thick
] coordinates {};

\end{axis}

\begin{axis}[
    name=precision,
    at={(dice.east)},
    anchor=west,
    xshift=1cm,
    clip=false,
    boxplot/draw direction=x,
    height=5cm,
    width=0.4\linewidth,
    title={Precision},
    xmin=0.8, xmax=1.0,
    xtick={0.8, 0.85, 0.9, 0.95, 1.0},
    ytick={1, 2, 3, 4, 5, 6},
    yticklabels={}
]

\addplot+[
    boxplot prepared={
        lower whisker=0.8541176319122314,
        lower quartile=0.8637364208698273,
        median=0.8793109357357025,
        upper quartile=0.8930473029613495,
        upper whisker=0.9085304737091064
    },
    black,
    solid,
    thick
] coordinates {};

\addplot+[
    boxplot prepared={
        lower whisker=0.8502352833747864,
        lower quartile=0.8589302897453308,
        median=0.8669086396694183,
        upper quartile=0.8728009909391403,
        upper whisker=0.9078103303909302
    },
    black,
    solid,
    thick
] coordinates {};

\addplot+[
    boxplot prepared={
        lower whisker=0.8995004892349243,
        lower quartile=0.9229026436805725,
        median=0.9314559102058411,
        upper quartile=0.9388822764158249,
        upper whisker=0.9418632984161377
    },
    OliveGreen,
    solid,
    thick
] coordinates {};

\addplot+[
    boxplot prepared={
        lower whisker=0.9319875240325928,
        lower quartile=0.9392609596252441,
        median=0.9405905902385712,
        upper quartile=0.9481289982795715,
        upper whisker=0.9621942043304443
    },
    black,
    solid,
    thick
] coordinates {};

\addplot+[
    boxplot prepared={
        lower whisker=0.892531693,
        lower quartile=0.92802833,
        median=0.9305585026741028,
        upper quartile=0.9360301047563553,
        upper whisker=0.9421557784080505
    },
    black,
    solid,
    thick
] coordinates {};

\addplot+[
    boxplot prepared={
        lower whisker=0.9120359420776367,
        lower quartile=0.9281624853610992,
        median=0.9466430246829987,
        upper quartile=0.9543877393007278,
        upper whisker=0.9563491940498352
    },
    OliveGreen,
    solid,
    thick
] coordinates {};

\node[black] at (axis cs:1.06,7.91){Parameters};
\node[OliveGreen] at (axis cs:1.06,6){2658};
\node[black] at (axis cs:1.06,5){2560};
\node[black] at (axis cs:1.06,4){2560};
\node[OliveGreen] at (axis cs:1.06,3){2564};
\node[black] at (axis cs:1.06,2){2622};
\node[black] at (axis cs:1.06,1){2550};

\end{axis}

\end{tikzpicture}
\caption{Comparison of PDE-(G-)CNNs tested on Lines. The new architectures based on RDS filtering are highlighted in \textcolor{OliveGreen}{green}.}
\label{fig:trained_inpainting}
\end{figure*}

\subsubsection{Denoising}\label{sec:trained_denoising}
We have adapted the DRIVE dataset \cite{Staal2004RidgeRetina} as a denoising task.\footnote{Currently available at \url{https://drive.grand-challenge.org/}.} DRIVE consists of colour images of the retina, paired with vessel segmentation masks. Typically, the goal is to predict the vessel segmentation mask from the retinal image. For our denoising task, we ignore the masks, and add random correlated white noise as in Sec.~\ref{sec:denoising}: now the standard deviation is $\rho = 2$ and the power of the noise is $\sigma = 0.25$.\footnote{The images have been normalised to take values in $[0, 1]$, as is common in machine learning tasks.} The dataset contains 20 training images of $584 \times 565$ pixels, which we divide into patches of $64 \times 64$ pixels overlapping by 16 pixels. After removing patches without any vessels, we are left with 2409 patches. Every time a training patch is encountered, new correlated white noise is generated. The dataset additionally contains 20 test images of $584 \times 565$ pixels, which are evaluated whole. 

We optimise the Mean Squared Error (MSE), and then quantify the performance in terms of PSNR (Eq.~\eqref{eq:psnr}) and Structural Similarity Index Measure (SSIM) \cite{Wang2004ImageSimilarity}. SSIM was created as a measure of perceptual image similarity, as opposed to pixel-wise errors.

To make the comparison fair, we keep the network complexity (approximately) constant. All networks have 6 layers. The CDE and DDE PDE-G-CNNs have a width of 16 channels; the RDS PDE-G-CNN has a width of 15 channels; the CDE and DDE PDE-CNNs have a width of 17 channels; and the RDS PDE-CNN has a width of 16 channels. In this way every network has $\sim 4200$ trainable parameters.

The quantitative results are summarised in Fig.~\ref{fig:trained_denoising}. Notably, the spread in the performance of a given architecture is rather large, which suggests that the optimisation hyperparameters could be improved. We have also visualised a sample set of outputs from the networks that achieved the best performance on the test set Figs.~\ref{fig:drive}c-h. Overall, the PDE-G-CNNs perform better than the PDE-CNNs, particularly in terms of PSNR. This is reflected in the visual performance: all networks manage to reconstruct the main vascular structure, but the PDE-CNNs produce noticeably softer outputs. We also see that the RDS PDE-G-CNN can outperform the DDE and CDE PDE-G-CNNs in terms of PSNR, with again a comparatively small improvement in SSIM.

\subsubsection{Inpainting}\label{sec:trained_inpainting}
We use the Lines dataset \cite{Bellaard2023Analysis} as an inpainting task. The goal of the Lines dataset is the following: given an input image of seemingly randomly placed line segments (Fig.~\ref{fig:lines}a), complete the line and remove the spurious line segments (Fig.~\ref{fig:lines}b). The dataset consists of 512 pairs of inputs and segmentations for training and validation, and 128 pairs for testing.

The images are normalised to take values in the range $[0, 1]$.\footnote{This can be achieved by passing the output of the networks through a sigmoid function.}  We quantify the performance using the Dice coefficient and the precision, and optimise a continuous Dice loss; more details can be found in App.~\ref{app:inpainting_metrics}.

To make for a fair comparison, we keep the network complexity (approximately) constant. All networks have 6 layers. The CDE and DDE PDE-G-CNNs have a width of 16 channels; the RDS PDE-G-CNN has a width of 15 channels; the CDE and DDE PDE-CNNs have a width of 18 channels; and the RDS PDE-CNN has a width of 17 channels. In this way every network has $\sim 2600$ trainable parameters.

The results are summarised in Fig.~\ref{fig:trained_inpainting}. We first note that the PDE-G-CNNs outperform the PDE-CNNs. This could be due to the improved expressivity permitted by processing in $\M$ instead of $\Rtwo$. For the PDE-G-CNNs, all PDEs perform approximately equally well. However, the RDS PDE-CNN significantly outperforms both the CDE and DDE PDE-CNNs, and approaches the performance of the PDE-G-CNNs. Expressivity may again play a role. Recall that we constrained the networks to use isotropic metrics for the sake of equivariance. However, the RDS PDE layer does not act isotropically due to the gating, which could be advantageous in this highly anisotropic inpainting task. Since PDE-CNNs are in general less computationally demanding than PDE-G-CNNs, this could make the RDS PDE-CNN an interesting alternative to PDE-G-CNNs in compute constrained environments.

We have also visualised a sample set of outputs Figs.~\ref{fig:lines}c-h. Since it is only a single sample (out of 740), we cannot draw strong conclusions from the comparison. However, it does show where typical errors that are made. For instance, we see that three of the networks incorrectly identify a line in the top-right corner, due to the presence of a pair of line segments that are coincidentally aligned. 

\section{Conclusion}\label{sec:conclusion}
In this article, we developed and investigated RDS filtering on position-orientation space $\M$, which preserves crossings unlike RDS filtering on $\Rtwo$. For this, we use generalised Laplacians that are induced by Lie-Cartan connections instead of the Levi-Civita connection. We showed how these Lie-Cartan Laplacians differ from the standard Laplace-Beltrami operator (Thms.~\ref{thm:left_invariant_lie_cartan_laplacian} and \ref{thm:data_driven_lie_cartan_laplacian}), yet still generate analytic evolutions (Thm.~\ref{thm:data_driven_lie_cartan_laplacian_sectorial} and Cor.~\ref{cor:lie_cartan_laplacian_sectorial}). We additionally prove our scheme satisfies a maximum-minimum principle (Thm.~\ref{thm:stability}). Subsequently, we showed experimentally that RDS filtering can outperform existing algorithms (TR-TV \cite{smets2021tvflow}, BM3D \cite{bm3dold}, NLM \cite{nlm}, $\Rtwo$ RDS \cite{schaefer2024regularisedds}) on denoising tasks in terms of maximal PSNR and sensitivity to stopping time, cf. Fig.~\ref{fig:plots}. Additionally, $\M$ RDS filtering is capable of inpainting crossing structures, unlike $\Rtwo$ RDS filtering, cf. Fig,~\ref{fig:inpainting}.

One limitation with RDS filtering, shared with many PDE-based image processing techniques, is that it requires manually optimising numerous parameters. We therefore integrated it into the PDE-based Convolutional Neural Network framework \cite{Smets2022PDEGCNNs}, creating both PDE-CNNs on $\Rtwo$ and PDE-G-CNNs on $\M$. In these networks, the PDE parameters are learned from data. We compare RDS PDE-(G-)CNNs to existing PDE-(G-)CNNs at equal network complexity on a denoising and an inpainting task. We find that the RDS PDE-CNN significantly outperforms the other PDE-CNNs on the inpainting task, approaching the performance of the PDE-G-CNNs. We hypothesise that this is due to the highly anisotropic nature of the contour completion task: like the PDE-G-CNNs, but unlike the other PDE-CNNs, the RDS PDE-CNN layer can simultaneously be anisotropic and equivariant. This could make the RDS PDE-CNN an interesting, more computationally affordable, alternative to existing PDE-G-CNNs.

\textbf{Future work:} One advantage of PDE-G-CNNs compared to other machine learning frameworks is their improved interpretability, due in part to the fact that the used PDEs are well-known from classical image processing. It would therefore be interesting to better understand the RDS PDE, for instance by examining whether it is a gradient flow like e.g. TR-TV \cite{smets2021tvflow}. Additionally, the accuracy of our proposed trainable RDS method should be investigated, to ensure that intuitions from the classical PDE method can be carried over to the new machine learning method.

\backmatter

\bmhead{Acknowledgements}
The Dutch Research Council (NWO) is gratefully acknowledged for its financial support via \href{https://www.nwo.nl/projecten/vic202031}{VIC.202.031}. Furthermore, the Eindhoven AI Systems Institute (EAISI) is gratefully acknowledged for financial support through the \href{https://www.tue.nl/en/research/institutes/eindhoven-artificial-intelligence-systems-institute/ai-research/eaisi-emdair-program/#c392690}{EMDAIR programme}.

We gratefully acknowledge B.M.N. Smets for developing \& maintaining the open source package LieTorch in which we integrated the new RDS PDEs, and for providing valuable input on the gating mechanism. 

\bmhead{Availability of Data and Code}
The code and data for the experiments in Sec.~\ref{sec:experiments} are available at \url{https://github.com/finnsherry/M2RDSFiltering}. The code for the experiments in Sec.~\ref{sec:experiments_trained} is available at \url{https://gitlab.com/bsmetsjr/lietorch}. The DRIVE dataset \cite{Staal2004RidgeRetina} is available at \url{https://drive.grand-challenge.org/}. The Lines dataset is available from the authors of \cite{Bellaard2023Analysis} on request.

\section*{Declarations}
\bmhead{Conflict of interest} R. Duits is a member of the editorial board of Journal of Mathematical Imaging and Vision (JMIV).

\bibliography{sn-bibliography}\label{references}

\begin{appendices}

\section{Equivariant Processing}\label{sec:equivariance}
Taking into account symmetries has been an invaluable principle in myriad fields, not least image processing. Symmetries are actions that preserve structure, and a processing method preserves a symmetry if the two commute. This means that first performing a symmetry action and subsequently processing gives the same output as first processing and thereafter performing a corresponding symmetry action. We formalise this mathematically with the concept of \emph{equivariance}:
\begin{definition}[Equivariance]\label{def:equivariance}
Let $G$ be a Lie group, $X$ and $Y$ be Banach spaces, and $\U: G \to B(X)$ and $\V: G \to B(Y)$ be representations. An operator $\Phi: X \to Y$ is said to be \emph{equivariant} if it commutes with the representations, i.e. $\Phi \after \U_g = \V_g \after \Phi$ for $g \in G$.
\end{definition}
Of particular interest are equivariant operators on function spaces, since we can model e.g. images as functions on the plane. If $G$ acts on homogeneous space $M$, then it naturally induces the \emph{quasi-regular representation} $\U$ on $\Ltwo(M)$:
\begin{equation}\label{eq:quasi_regular_representation_general}
(\U_g f) (p) \coloneqq f(g^{-1} p),
\end{equation}
for all $g \in G$, $p \in M$, and $f \in \Ltwo(M)$. In classical image processing, our methods should respect symmetries of the plane. We focus on roto-translation symmetry, encoded in the special Euclidean group $\SE(2)$ (Def.~\ref{def:se2}).
$\SE(2)$ acts on the plane by $(\vec{x}, R) \vec{y} = \vec{x} + R \vec{y}$ for $(\vec{x}, R) \in \SE(2)$, $\vec{y} \in \Rtwo$, so that the induced quasi-regular representation $\U$ is given by
\begin{equation}\label{eq:quasi_regular_representation_r2}
(\U_{(\vec{x}, R)} f) (\vec{y}) \coloneqq f(R^{-1}(\vec{y} - \vec{x})),
\end{equation}
for $(\vec{x}, R) \in \SE(2)$, $\vec{y} \in \Rtwo$, and $f \in \Ltwo(\Rtwo)$.
Likewise, $\SE(2)$ acts on position-orientation space by $L_{(\vec{x}, R_\theta)} (\vec{y}, \phi) \coloneqq (R_\theta \vec{y} + \vec{x}_g, \phi + \theta)$ for $(\vec{x}, R_\theta) \in \SE(2)$, $(\vec{y}, \phi) \in \M$, inducing the regular representation
\begin{equation}\label{eq:quasi_regular_representation_m2}
\begin{split}
(\cL_{(\vec{x}, R_\theta)} f) (\vec{y}, \phi) & \coloneqq (f \after L_{(\vec{x}, R_\theta)}^{-1}) (\vec{y}, \phi) \\
& = f(R_\theta^{-1}(\vec{y} - \vec{x}), \phi - \theta),
\end{split}
\end{equation}
for $(\vec{x}, R) \in \SE(2)$, $(\vec{y}, \phi) \in \M$, and $f \in \Ltwo(\M)$.
If we are now performing processing on images on $\Rtwo$, we want it to commute with the quasi-regular representation $\U$. Intuitively, this means that if the input of the processing is roto-translated in some way, the output must be roto-translated accordingly. When we do multi-orientation processing (see Sec.~\ref{sec:m2}), in intermediate steps the data lives on $\M$, and so the processing between those steps should commute with $\cL$. This can be achieved by working with e.g. invariant \eqref{eq:invariant_vector_field} or gauge frame \eqref{eq:first_gauge_vector} vector fields. Finally, the data must be lifted from $\Rtwo$ to $\M$ and projected back from $\M$ to $\Rtwo$ in an equivariant manner. This means that we must have
\begin{equation*}
\W_\psi \after \U_g = \cL_g \after \W_\psi, \textrm{ and } \Proj \after \cL_g = \U_g \after \Proj,
\end{equation*}
for all $g \in \SE(2)$. It is not hard to confirm that this is indeed the case for the orientation score transform and fast reconstruction formula (Def.~\ref{def:orientation_score}), for more details see \cite{Duits2005PerceptualAnalysis}.

\section{Computing the Gauge Frame}\label{app:computing_gauge_frame}
We here discuss how to compute the first gauge vector field $\A_1^U$; the others can be directly computed from $\A_1^U$ as explained in Sec.~\ref{sec:gauge_frames}. By Def.~\ref{def:first_gauge_vector}, the first gauge vector field is given by 
\begin{equation*}
\A_1^U|_p \coloneqq \underset{\substack{X \in T_p \M \\ \norm{X}_{\G_\xi} = 1}}{\argmin} \norm*{\nabla_X^{[0]} \gradient_{\G_\xi}|_p U}_{\G_\xi}^2,
\end{equation*}
where $\nabla_\cdot^{[0]}$ is the $0$ Lie-Cartan connection (Def.~\ref{def:lie_cartan_connection}). We start, using the properties of the $0$ Lie-Cartan connection and the definition of the Riemannian gradient, by expanding the Hessian term in Eq.~\eqref{eq:first_gauge_vector}:
\begin{equation*}
\begin{split}
\nabla_X^{[0]} \gradient_{\G_\xi}|_p U & = X^k \nabla_{\A_k}^{[0]} (g^{ij} \A_i|_p U \A_j|_p) \\
& = X^k g^{ij} (\A_k \A_i|_p U) \A_j|_p.
\end{split}
\end{equation*}
Now, we note that the metric $\G_\xi$ is diagonal, with components $g_{11} = \xi^2 = g_{22}$, $g_{33} = 1$, and $g_{ij} = 0$ for $i \neq j$. We define $\xi_j \coloneqq \sqrt{g_{jj}}$ as a short-hand. We can then write out the norm in \eqref{eq:first_gauge_vector} as
\begin{equation*}
\begin{split}
& \norm*{\nabla_X^{[0]} \gradient_{\G_\xi}|_p U}_{\G_\xi}^2 = \norm*{\sum_{i,j,k} X^k g^{ij} (\A_k \A_i|_p U) \A_j|_p}_{\G_\xi}^2 \\
& = \norm*{\sum_{i,j} X^i \xi_j^{-2} (\A_i \A_j|_p U) \A_j|_p}_{\G_\xi}^2 \\
& = \sum_j \xi_j^2 \left(\sum_i X^i \xi_j^{-2} (\A_i \A_j|_p U)\right)^2.
\end{split}
\end{equation*}
Since the expressions aren't linear, we no longer use the Einstein summation convention and instead sum explicitly. We may suggestively rewrite the norm as
\begin{multline*}
\norm*{\nabla_X^{[0]} \gradient_{\G_\xi}|_p U}_{\G_\xi}^2 = \\ \sum_j \left(\sum_i \A_i \A_j|_p U X^i\right) \xi_j^{-2} \left(\sum_i \A_i \A_j|_p U X^i\right).
\end{multline*}
We now define the matrix field $\mat{H}$ with the components $H_j^i \coloneqq \A_j \A_i U$,
with the upper and lower indices corresponding to the rows and columns, respectively. We furthermore define
$\mat{M}_\xi \coloneqq \operatorname{diag}(\xi, \xi, 1)$ and $\vec{X} \coloneqq (X^1, X^2, X^3)^T$, where $X^i$ is the the $i$-th component of $X$ with respect to the invariant frame. We now note that 
\begin{equation*}
\sum_i \A_i \A_j|_p U X^i = \sum_i H_i^j X^i = [\mat{H} \vec{X}]^j,
\end{equation*}
so that 
\begin{equation*}
\begin{split}
& \sum_j \left(\sum_i \A_i \A_j|_p U X^i\right) \xi_j^{-2} \left(\sum_i \A_i \A_j|_p U X^i\right) \\
& = (\mat{H} \vec{X})^T \mat{M}_\xi^{-2} (\mat{H} \vec{X}) \\
& = \norm{\mat{M}_\xi^{-1} \mat{H} \vec{X}}_2^2.
\end{split}
\end{equation*}
Finally, we may rewrite the constraint:
\begin{equation*}
\norm{X}_{\G_\xi}^2 = \norm{\mat{M}_\xi \vec{X}}_2^2.
\end{equation*}
Therefore, we see that Problem \eqref{eq:first_gauge_vector} is equivalent to the following optimisation problem:
\begin{equation}\label{eq:first_gauge_vector_svd}
\textrm{minimise } \norm{\mat{M}_\xi^{-1} \mat{H} \vec{X}}_2^2, \textrm{ s.t. } \norm{\mat{M}_\xi \vec{X}}_2^2 = 1.
\end{equation}
We can now solve Problem \eqref{eq:first_gauge_vector_svd} with the Lagrange multiplier method, with Lagrangian
\begin{equation*}
\cL(\vec{X}; \lambda) \coloneqq \norm{\mat{M}_\xi^{-1} \mat{H} \vec{X}}_2^2 + \lambda (1 - \norm{\mat{M}_\xi \vec{X}}_2^2).
\end{equation*}
Differentiating to the components of $\vec{X}$ and $\lambda$, we find
\begin{multline*}
\nabla_{\vec{X}, \lambda} \cL(\vec{X}; \lambda) = \\ \begin{pmatrix}
2 \mat{H}^T \mat{M}_\xi^{-1} (\mat{M}_\xi^{-1} \mat{H} \vec{X}) - 2 \lambda \mat{M}_\xi (\mat{M}_\xi \vec{X}) \\
1 - \vec{X}^T \mat{M}_\xi^{2} \vec{X}
\end{pmatrix} = \vec{0}.
\end{multline*}
Cleaning up the derivatives with respect to the components of $\vec{X}$ we get
\begin{multline*}
2 \mat{H}^T \mat{M}_\xi^{-1} (\mat{M}_\xi^{-1} \mat{H} \vec{X}) = 2 \lambda \mat{M}_\xi (\mat{M}_\xi \vec{X}) \\
\iff \mat{H}^T \mat{M}_\xi^{-2} \mat{H} \vec{X} = \lambda \mat{M}_\xi^2 \vec{X} \\
\iff \mat{M}_\xi^{-1} \mat{H}^T \mat{M}_\xi^{-2} \mat{H} \mat{M}_\xi^{-1} \tilde{\vec{X}} = \lambda \tilde{\vec{X}},
\end{multline*}
where $\tilde{\vec{X}} \coloneqq \mat{M}_\xi \vec{X}$. We can now interpret $\tilde{\vec{X}}$ as a singular vector of $\mat{M}_\xi^{-1} \mat{H} \mat{M}_\xi^{-1}$. We can additionally interpret $\tilde{\vec{X}}$ as a sort of dimensionless version of $\vec{X}$. It is then evident that Problem \eqref{eq:first_gauge_vector_svd} is
solved by $\vec{X} = \mat{M}_\xi^{-1} \tilde{\vec{X}}$ where $\tilde{\vec{X}}$ is the singular vector corresponding to the smallest singular value $\lambda_\textrm{min}$ of $\mat{M}_\xi^{-1} \mat{H} \mat{M}_\xi^{-1}$, moreover satisfying $\norm{\mat{M}_\xi \vec{X}}_2^2 = 1$.

In practice, we also regularise the matrix field $\mat{H}$ componentwise, i.e. we instead define
\begin{equation*}
H_j^i \coloneqq G_\alpha *_{\SE(2)} (\A_j \A_i U),
\end{equation*}
where $G_\alpha$ is a spatially isotropic Gaussian kernel with scale $\alpha > 0$ and $*_{\SE(2)}$ is the group convolution on $\M$. This helps smooth out the gauge frame.

Finally, we prove that the gauge frame is equivariant.
\begin{proposition}[Equivariance of Gauge Frame]\label{prop:equivariance_gauge_frame}
The gauge frame $\{\A_i^\cdot\}_i: C^\infty(\M) \to \sections(T\M)^3$ is equivariant, i.e.
\begin{equation}\label{eq:equivariance_gauge_frame}
\A_i^{\cL_g U}|_p = (L_g)_* \A_i^U|_{L_g^{-1} p},
\end{equation}
for all $U \in C^\infty(\M)$, $g \in \SE(2)$ $p \in \M$, $i = 1, 2, 3$.
\end{proposition}
\begin{proof}
Note that for all $i, j$, $H_j^i: C^\infty(\M) \to C^\infty(\M)$ is equivariant, since for all $U \in C^\infty(\M)$ and $g \in \SE(2)$
\begin{equation*}
G_\alpha *_{\SE(2)} (\A_j \A_i \cL_g U) = \cL_g G_\alpha *_{\SE(2)} (\A_j \A_i U),
\end{equation*}
due to the invariance of $\A_i$ and $\A_j$ and the equivariance of $\SE(2)$ convolutions. Consequently, the matrix
\begin{equation*}
\mat{M}_\xi^{-1} \mat{H}^T \mat{M}_\xi^{-2} \mat{H} \mat{M}_\xi^{-1}: C^\infty(\M) \to C^\infty(\M, \R^{3 \times 3})
\end{equation*}
is equivariant, so that also $\vec{X}: C^\infty(M) \to C^\infty(M, \R^3)$, the vector corresponding to the smallest singular value of this matrix, is equivariant. It then follows that $\A_1^\cdot$ is equivariant, since for all $U \in C^\infty(\M)$, $g \in \SE(2)$, and $p \in \M$
\begin{equation*}
\begin{split}
\A_1^{\cL_g U} |_p & = X_{\cL_g U}^i(p) \A_i |_p = \cL_g X_U^i(p) \A_i |_p \\
& = X_U^i(L_g^{-1} p) (L_g)_* \A_i |_{L_g^{-1} p} \\
& = (L_g)_* X_U^i(L_g^{-1} p) \A_i |_{L_g^{-1} p} = (L_g)_* \A_1^U |_{L_g^{-1} p}.
\end{split}
\end{equation*}
The equivariance of $\A_2^\cdot, \A_3^\cdot: C^\infty(\M) \to \sections(T\M)$ then follows from the fact that selecting perpendicular vectors is also equivariant as the metric is invariant.
\end{proof}

\section{Riemannian Divergence in terms of Levi-Civita Connection}\label{app:divergence_connection}
It is a known result that on a Riemannian manifold $(M, \G)$ it holds that 
\begin{equation*}
\divergence_\G(X) = \trace(\nabla_\cdot^\mathrm{LC} X)
\end{equation*}
for all $X \in \sections(T M)$, see e.g. \cite[Problem~5-14]{lee2018riemannian}. For completeness, we give a proof of this result. For this, we will evaluate both expressions at each point $p \in M$ in normal coordinates, i.e. coordinates $(x^i)$ such that the connection symbols of the Levi-Civita connection with respect to the coordinate frame $\{\partial_i\}_i$ vanish at $p$.

Since $\{\partial_i\}_i$ form a coordinate frame, the divergence is given by
\begin{equation*}
\divergence_\G (X) = \frac{1}{\sqrt{\det(g)}} \partial_i (\sqrt{\det(g)} X^i).
\end{equation*}
In normal coordinates, the first partial derivatives around $p$ of the components of the metric vanish. Using the chain rule we can then see that the first partial derivatives around $p$ of the determinant then also vanish:
\begin{equation*}
\partial_i g = \frac{\partial \det(g)}{\partial g_{kl}} \partial_i g_{kl} = 0.
\end{equation*}
Consequently we may conclude that
\begin{equation*}
\divergence_\G (X) = \frac{1}{\sqrt{\det(g)}} \partial_i (\sqrt{\det(g)} X^i) = \partial_i X^i.
\end{equation*}
The trace of the covariant derivative is given by
\begin{equation*}
\begin{split}
\trace(\nabla_\cdot^\mathrm{LC} X) & = \langle \diff x^i, \nabla_{\partial_i}^\mathrm{LC} (X^j \partial_j) \rangle \\
& = \langle \diff x^i, (\partial_i X^j) \partial_j + X^j \Gamma_{ij}^k \partial_k \rangle \\
& = \partial_i X^i + \Gamma_{ij}^i X^j.
\end{split}
\end{equation*}
Since we are using normal coordinates, at $p$ this reduces to
\begin{equation*}
\trace(\nabla_\cdot^\mathrm{LC}|_p X) = \partial_i|_p X^i.
\end{equation*}
Therefore,
\begin{equation*}
\divergence_\G|_p(X) = \trace(\nabla_\cdot^\mathrm{LC}|_p X)
\end{equation*}
for all $p \in M$, as required.

Note that $\trace(\nabla_\cdot^\mathrm{LC} X)$ does not depend on the orientability of $M$, unlike the typical definition of divergence in terms of the Riemannian volume form.

\section{Properties of Gauge Frame Lie-Cartan Connections}\label{app:properties_lie_cartan}
We here show that the $0$ gauge frame Lie-Cartan connection is metric compatible with any data-driven invariant metric.
\begin{proposition}
Let $G$ be a Lie group, with gauge frame $\{\A_i^U\}_i$, and let $\G^U$ be a data-driven invariant metric thereon. Then the $\nu$ gauge frame Lie-Cartan connection $\nabla^{[\nu], U}$ is metric compatible if $\nu = 0$.
\end{proposition}
\begin{proof}
Since $\G^U$ is a data-driven invariant metric, we have
\begin{equation*}
\G^U(\A_i^U, \A_j^U) = g_{ij}
\end{equation*}
for constants $g_{ij} \in \R$. Hence, we see that
\begin{equation*}
\A_i^U \G^U(\A_j^U, \A_k^U) = \A_i^U g_{jk} = 0.
\end{equation*}
We also have
\begin{equation*}
\begin{split}
\G^U(\nabla_{\A_i^U}^{[\nu], U} \A_j^U, \A_k^U) & = \G^U(\nu [\A_i^U, \A_j^U], \A_k^U) \\
& \eqqcolon \nu d_{ij}^l \G^U(\A_l^U, \A_k^U) =  \nu d_{ij}^l g_{lk}.
\end{split}
\end{equation*}
Likewise,
\begin{equation*}
\G^U(\A_j^U, \nabla_{\A_i^U}^{[\nu], U} \A_k^U) = \nu d_{ik}^l g_{jl}.
\end{equation*}
Hence, metric compatibility comes down to
\begin{equation*}
0 = \nu (d_{ij}^l g_{lk} + d_{ik}^l g_{jl}),
\end{equation*}
which is certainly satisfied when $\nu = 0$.
\end{proof}
In practice, the gauge frame Lie-Cartan connection will only be metric compatible if $\nu = 0$, since the structure functions $d_{ij}^k$ can vary wildly and will not cancel out. Note that there is an analogous result for the standard $0$ Lie-Cartan connection with invariant metrics \cite[Cor.~2]{duits2021cartanconnection}.

\section{Proof of Stability\texorpdfstring{ Thm.~\ref{thm:stability}}{}}\label{app:proof_stability}
We here prove Thm.~\ref{thm:stability}.
\begin{theorem}[Stability of RDS filtering on \texorpdfstring{$\M$}{position-orientation space}]
Choose timestep $\tau \leq \min\{\tau_D, \tau_S\}$, where
\begin{equation}
\begin{split}
\tau_D^{-1} & \coloneqq 2 \left(\frac{g_D^{11} + g_D^{22}}{\Delta_{xy}^2} + \frac{g_D^{33}}{\Delta_\theta^2}\right), \textrm{ and } \\
\tau_S^{-1} & \coloneqq \sqrt{\frac{g_M^{11} + g_M^{22}}{\Delta_{xy}^2} + \frac{g_S^{33}}{\Delta_\theta^2}}.
\end{split}
\end{equation}
Then, this scheme satisfies a maximum-minimum principle:
\begin{equation}
\min_{w, h, o} U_{w, h, o}^0 \leq U_{i, j, k}^t \leq \max_{w, h, o} U_{w, h, o}^0,
\end{equation}
for all $i, j, k, t$, where $U^t$ is solution at time step $t \in \N_{\geq 0}$.
\end{theorem}
\begin{proof}
The proof is analogous to that of the stability result for RDS filtering on $\Rtwo$ \cite[Thm.~1]{schaefer2024regularisedds}. We show that 
\begin{equation*}
\min_{w, h, o} U_{w, h, o}^t \leq U_{i, j, k}^{t + 1} \leq \max_{w, h, o} U_{w, h, o}^t \textrm{ for all } i, j, k;
\end{equation*}
\eqref{eq:stability} then follows by induction. The discretised RDS PDE is given by 
\begin{equation*}
\begin{split}
& U_{i, j, k}^{t + 1} \coloneqq U_{i, j, k}^t + \tau \\
& \cdot \Big(g_{i, j, k}^t \cdot (\laplace_{\G_D} U)_{i, j, k}^t \\
& - (1 - g_{i, j, k}^t) \cdot S_{i, j, k}^t \cdot (\norm{\gradient_{\G_M} U}_{\G_M})_{i, j, k}^t \Big),
\end{split}
\end{equation*}
where
\begin{equation*}
\begin{split}
g_{i, j, k}^t & \coloneqq g(\norm{\gradient_{\G_g} U_\nu})_{i, j, k}^t, \textrm{ and } \\
S_{i, j, k}^t & \coloneqq S_\rho(\laplace_{\G_S}^\perp U_\sigma)_{i, j, k}^t.
\end{split}
\end{equation*}
The Laplacian term is given by
\begin{equation*}
\begin{split}
(\laplace_{\G_D} U)_{i, j, k}^t & \coloneqq \frac{g_D^{11}}{\Delta_{xy}^2} \Delta_{\cos_k, \sin_k, 0} U_{i, j, k}^t \\
& + \frac{g_D^{22}}{\Delta_{xy}^2} \Delta_{-\sin_k, \cos_k, 0} U_{i, j, k}^t \\
& + \frac{g_D^{33}}{\Delta_{\theta}^2} \Delta_{0, 0, 1} U_{i, j, k}^t,
\end{split}
\end{equation*}
with centred differences
\begin{equation*}
\Delta_{a, b, c}^c U_{i, j, k}^t \coloneqq U_{i + a, j + b, k + c}^t - 2 U_{i, j, k}^t + U_{i - a, j - b, k - c}^t
\end{equation*}
and $\cos_k \coloneqq \cos(\theta_k)$ and $\sin_k \coloneqq \sin(\theta_k)$, and where off-grid samples (e.g. $U_{i + \cos_k, j + \sin_k, k}^t$) are computed by trilinear interpolation. Since trilinear interpolation is min-max stable, we can see that 
\begin{multline*}
\min_{w, h, o} U_{w, h, o}^t - U_{i, j, k}^t \leq \tau (\laplace_{\G_D} U)_{i, j, k}^t \\ \leq \max_{w, h, o} U_{w, h, o}^t - U_{i, j, k}^t.
\end{multline*}
For the morphological term we use the following Rouy-Tourin style discretisation \cite{rouy1992viscosity,bekkers2015subriemanniangeodesics}. In areas where dilation is performed (i.e. $S_{i, j, k}^t < 0$), the discretisation is given by
\begin{equation*}
\begin{split}
((\norm{\gradient_{\G_M} U}_{\G_M})_{i, j, k}^t)^2 & \coloneqq \frac{g_M^{11}}{\Delta_{xy}^2} (\Delta_{\cos_k, \sin_k, 0}^u U_{i, j, k}^t)^2 \\
& + \frac{g_M^{22}}{\Delta_{xy}^2} (\Delta_{-\sin_k, \cos_k, 0}^u U_{i, j, k}^t)^2 \\
& + \frac{g_M^{33}}{\Delta_{xy}^2} (\Delta_{0, 0, 1}^u U_{i, j, k}^t)^2 \\
& \leq \left(\frac{g_M^{11} + g_M^{22}}{\Delta_{xy}^2} + \frac{g_S^{33}}{\Delta_\theta^2}\right) \\
&\cdot (\max_{w, h, o} U_{w, h, o}^t - U_{i, j, k}^t)^2,
\end{split}
\end{equation*}
with upwind differences
\begin{multline*}
\Delta_{a, b, c}^u U_{i, j, k}^t \coloneqq \\ \max\{U_{i + a, j + b, k + c}^t - U_{i, j, k}^t, U_{i, j, k}^t - U_{i - a, j - b, k - c}^t, 0\}
\end{multline*}
It can analogously be shown that in the case erosion is performed ($S_{i, j, k}^t > 0$), the term is bounded from below by
\begin{multline*}
((\norm{\gradient_{\G_M} U}_{\G_M})_{i, j, k}^t)^2 \geq \\ \left(\frac{g_M^{11} + g_M^{22}}{\Delta_{xy}^2} + \frac{g_S^{33}}{\Delta_\theta^2}\right) (U_{i, j, k}^t - \min_{w, h, o} U_{w, h, o}^t)^2.
\end{multline*}
Since $-1 \leq S_{i, j, k}^t \leq 1$, we can then see that 
\begin{multline*}
U_{i, j, k}^t - \min_{w, h, o} U_{w, h, o}^t \leq \tau -S_{i, j, k}^t (\norm{\gradient_{\G_M} U}_{\G_M})_{i, j, k}^t \\ \leq \max_{w, h, o} U_{w, h, o}^t - U_{i, j, k}^t.
\end{multline*}
It then follows that
\begin{equation*}
\begin{split}
U_{i, j, k}^{t + 1} & \leq U_{i, j, k}^t + g_{i, j, k}^t \cdot (\laplace_{\G_D} U)_{i, j, k}^t + (1 - g_{i, j, k}^t) \\
&  \cdot -S_{i, j, k}^t (\norm{\gradient_{\G_M} U}_{\G_M})_{i, j, k}^t \\
& \leq U_{i, j, k}^t + \max_{w, h, o} U_{w, h, o}^t - U_{i, j, k} = \max_{w, h, o} U_{w, h, o}^t,
\end{split}
\end{equation*}
and likewise $U_{i, j, k}^{t + 1} \geq \min_{w, h, o} U_{w, h, o}$, as required.
\end{proof}
The proof can be adapted to work in the gauge frame case as well. Since the gauge frame is normalised with respect to $\G_\xi$ (see Def.~\ref{def:first_gauge_vector}), while the invariant frame is normalised with respect to $\G_1$, the timestep must be changed to $\tau \leq \min\{\tau_D, \tau_S\}$, where
\begin{equation*}
\begin{split}
\tau_D^{-1} & \coloneqq 2 (g_D^{11} + g_D^{22} + g_D^{33}) / h^2, \textrm{ and } \\
\tau_S^{-1} & \coloneqq \sqrt{g_M^{11} + g_M^{22} + g_S^{33}} / h
\end{split}
\end{equation*}
for $h \coloneqq \min\{\xi \cdot \Delta_{xy}, \Delta_\theta\}$. We change our finite differences accordingly, e.g. the approximation of the terms in the Laplacian becomes
\begin{multline*}
(\A_i^U)^2 U_{i,j,k}^t \\
\approx \frac{1}{(\xi \Delta_{xy})^2} \Delta_{\xi X^1, \xi X^2, 0} U_{i,j,k}^t + \frac{1}{\Delta_\theta^2} \Delta_{X^3} U_{i,j,k}^t,
\end{multline*}
with $\A_i^U = X^j \A_j$. Then the min-max stability again follows from the stability of trilinear interpolation.

\section{Inpainting Metrics}\label{app:inpainting_metrics}
Here we provide additional details on the metrics we use to train and quantify the performance of the networks on the inpainting task in Sec.~\ref{sec:trained_inpainting}.

The performance of the networks can be quantified using the confusion matrix, which counts the number of 
\begin{itemize}
\item True Positives: $\TP(f, g) = \sum_\vec{x} \indicator\{f(\vec{x}) = 1 = g(\vec{x})\}$,
\item True Negatives: $\TN(f, g) = \sum_\vec{x} \indicator\{f(\vec{x}) = 0 = g(\vec{x})\}$,
\item False Positives: $\FP(f, g) = \sum_\vec{x} \indicator\{f(\vec{x}) = 1, g(\vec{x}) = 0\}$, and
\item False Negatives: $\FN(f, g) = \sum_\vec{x} \indicator\{f(\vec{x}) = 0, g(\vec{x}) = 1\}$,
\end{itemize}
with prediction $f$ and ground truth $g$.
We use the confusion matrix to compute the Dice coefficient, a commonly used segmentation metric \cite{Dice1945MeasuresSpecies}:
\begin{multline}\label{eq:dice}
\Dice(f, g) \coloneqq \\ \frac{2 * \TP(f, g) + \epsilon}{2 * \TP(f, g) + \FP(f, g) + \FN(f, g) + \epsilon}, 
\end{multline}
where $\epsilon \ll 1$ ensures the Dice coeffient remains well-defined when $f = 0 = g$.
Notably, this metric requires binary images, i.e. taking values in $\{0, 1\}$. This makes it unsuitable for training the networks, since this would involve thresholding, which is a discontinuous operation that is not compatible with our optimisation method. We instead optimise a continuous Dice loss: 
\begin{equation}\label{eq:inpainting_loss}
\cL_{\Dice}(f, g) = 1 - \frac{\sum_\vec{x} f(\vec{x}) \cdot g(\vec{x})}{\sum_p (f(\vec{x}) + g(\vec{x})) + \epsilon};
\end{equation}
it is not hard to verify that $\Dice(f, g) = 1 - \cL_{\Dice}(f, g)$ for binary $f, g$.
We additionally report the precision:
\begin{equation}\label{eq:precision}
\Precision(f, g) \coloneqq \frac{\TP(f, g) + \epsilon}{\TP(f, g) + \FP(f, g) + \epsilon}.
\end{equation}

\end{appendices}

\end{document}